\documentclass[smallextended]{svjour3}       % onecolumn (second format)
\smartqed  % flush right qed marks, e.g. at end of proof
\usepackage{graphicx}
\usepackage{dsfont}
\usepackage{color,xcolor,multirow}

 \newcommand{\h}[1]{\mathbf{#1}}
 \usepackage{graphicx,url}
\usepackage{amsmath,enumerate}
 \usepackage{amssymb}
 \usepackage{bm}
 \usepackage{algorithmic,algorithm}

\usepackage{url}
\usepackage{diagbox}

\newcommand{\off}[1]{}
\DeclareMathOperator*{\argmin}{arg\,min}

\DeclareMathOperator{\Rbb}{\mathbb{R}}
\newcommand{\rmnum}[1]{\romannumeral #1}
\begin{document}

\title{Sorted $L_1/L_2$ Minimization for Sparse Signal Recovery %\thanks{Grants or other notes
}


\author{Chao Wang \and 
Ming Yan \and
Junjie Yu%etc.
}


\institute{C. Wang   \at
             Department of Statistics and Data Science, Southern University of Science and Technology, Shenzhen 518005, Guangdong Province, China  \\
             National Centre for Applied Mathematics Shenzhen, Shenzhen 518055, Guangdong Province, China \\
              \email{wangc6@sustech.edu.cn}         %  \\
           \and
           M. Yan \at
School of Data Science, The Chinese University of Hong Kong, Shenzhen (CUHK-Shenzhen), Shenzhen 518172, China \\
\email{yanming@cuhk.edu.cn}
\and 
J. Yu \at 
Department of Statistics and Data Science, Southern University of Science and Technology, Shenzhen 518005, Guangdong Province, China \\
\email{12132909@mail.sustech.edu.cn} 
}

\date{Received: date / Accepted: date}

\maketitle

\begin{abstract}
This paper introduces a novel approach for recovering sparse signals using sorted $L_1/L_2$ minimization. The proposed method assigns higher weights to indices with smaller absolute values and lower weights to larger values, effectively preserving the most significant contributions to the signal while promoting sparsity. We present models for both noise-free and noisy scenarios, and rigorously prove the existence of solutions for each case. To solve these models, we adopt a linearization approach inspired by the difference of convex functions algorithm. Our experimental results demonstrate the superiority of our method over state-of-the-art approaches in sparse signal recovery across various circumstances, particularly in support detection.
% This paper introduces a novel method for recovering sparse signals using sorted $L_1/L_2$ minimization. Our method assigns higher weights to indices with smaller absolute values and lower weights to larger values, thereby preserving the most significant contributions to the signal while promoting sparsity. We propose models for both noise-free and noisy scenarios and prove the existence of solutions for each. To solve these models, we utilize a linearization approach similar to the difference of convex functions algorithm. Our experiments demonstrate the superiority of our method over state-of-the-art approaches in sparse signal recovery in most circumstances, particularly in support detection.
\end{abstract}

\section{Introduction}\label{section:introduction}
The  recovery of sparse signals is a fundamental and critical issue in compressed sensing (CS) \cite{donoho2006compressed}. It involves identifying the solution with the smallest number of non-zero entries in a linear equation, particularly in high-dimensional scenarios. Regularization methods are widely used and effective techniques for this purpose, as they balance data accuracy with penalty terms. Regularization can fit a function appropriately to the given constraint, alleviating overfitting to some extent. In compressed sensing, sparse signals can also be seen as compressible.
%It involves identifying the solution with the least number of non-zero entries in a linear equation, especially in high-dimensional scenarios. One of the most widely used and effective techniques is to employ regularization methods, which balance data accuracy with penalty terms. In this way, the regularization can fit a function appropriately on the given constraint and alleviate over-fitting to some extent. In the field of compressed sensing, sparse signals also can be regarded as compressible. 
Retrieving the sparsest and most accurate solution from a linear constraint can be expressed as an optimization problem, as shown below:
\begin{equation}\label{eq1.1}
\min\limits_{\mathbf{x}\in \mathbb{R}^n} \|\h x\|_0 \quad \text{s.t.} \quad A\h x =\h b,
\end{equation}
where  $\left\|\h x \right\|_{0}$ denotes the number of nonzero entries in $\h x$ with $A\in \mathbb{R}^{m\times n}$ as the sensing matrix and $\h b\in \mathbb{R}^m$ as the measurement vector .
% On the other hand, from the statistics perspective, 
% the above minimization problem can be regarded as finding the most significant coefficient, where $A$ is a design matrix. 
The problem \eqref{eq1.1} is computationally intractable and falls into the category of NP-hard problems~\cite{natarajan95}, meaning it cannot be solved in polynomial time as the problem size grows. As a practical alternative, the $L_1$ penalty regularization can be used instead, which replaces the $L_0$ norm with its convex envelope. The resulting optimization problems are
\begin{equation}\label{eq1.2}
    \min\limits_{\mathbf{x}\in \mathbb{R}^n} \|\h x\|_1 \quad \text{s.t.} \quad A\h x = \h b,
\end{equation}
and  
\begin{equation}\label{eq1.3}
   \underset{ \mathbf{x}\in \mathbb{R}^{n}}{\min} \{\lambda \left\| \h x \right\|_1+\frac{1}{2}\left\|A\h x-\h b \right\|_{2}^{2}\}, 
\end{equation}
where $\lambda$ is the regularization parameter.
A significant breakthrough in compressed sensing theory was achieved through the development of the restricted isometry property (RIP)~\cite{candes2005decoding}. This property provides constraints on the approximation to an orthogonal matrix, ensuring that the solution of \eqref{eq1.2} is the optimal solution to the original minimization problem \eqref{eq1.1}. The resulting optimization problems, \eqref{eq1.2} and \eqref{eq1.3}, can be solved using various algorithms, such as gradient descent, iterative support detection (ISD) method \cite{wang2010sparse}, coordinate descent \cite{10.1007/s10107-015-0892-3}, and the alternative direction method of multipliers (ADMM)~\cite{doi:10.1137/060654797}.

While the $L_1$ minimization is computationally tractable, it has been noted in~\cite{fan2001variable} that it is biased towards large coefficients.
% , which leads to suboptimal performance compared with the $L_0$ sparse recovery problem.  
As a result, nonconvex relaxation techniques have emerged to give closer approximations to the $L_0$ norm. In addition to the $p$-th power of $L_p$ with $0< p < 1$~\cite{chartrand07}, other widely-used sparse models include the smoothly clipped absolute deviation (SCAD)~\cite{xie2009scad}, capped $L_1$~\cite{zhang2009multi,shen2012likelihood,louYX16}, minimax concave penalty (MCP)~\cite{10.1214/09-AOS729}, and the transformed $L_1$ (TL1)~\cite{lv2009unified,zhangX18}.

In recent years, there has been interest in the $L_{1}$-$L_{2}$ model, which is defined as $\|\mathbf{x}\|_{1}-\|\mathbf{x}\|_{2}$, due to its success in sparse signal recovery \cite{zibulevsky2010l1,lou2018fast}, particularly in scenarios where the sensing matrix has high coherence. One advantage of the $L_{1}$-$L_{2}$ penalty \cite{yinEX14} over $L_1$ is its unbiased characterization of one-sparse vectors.
% , as $\|\mathbf{x}\|_{1-2}=0$ if and only if $\|\mathbf{x}\|_{0} \leq 1$. 
However, as the number of leading entries (in magnitudes) increases, $L_{1}$-$L_{2}$ becomes biased and behaves similarly to $L_1$.
Furthermore, a scale-invariant functional based on the ratio of $L_1$ and $L_2$ norms, denoted as $L_1/L_2$, has been proposed \cite{rahimi2019scale}. This penalty function performs well in both high and low-coherence cases and has been shown to possess the sNSP (strong Null Space Property). The sNSP provides a stronger condition than the original NSP (Null Space Property) \cite{6811411} for ensuring that a matrix $A$ can generate a solution vector $\h x$ that is a local minimizer of $L_1/L_2$. Several recent works have demonstrated the advantage of this penalty function in various applications such as signal processing \cite{rahimi2019scale,wang2019accelerated}, medical imaging reconstruction \cite{wang2021limited}, and super-resolution \cite{wang2022minimizing}.

The $L_{t,1}$ metric, proposed by Hu \cite{hu2012fast}, is a closely related metric to this study. It excludes large magnitude entries in penalization, resulting in a better approximation to $L_0$ than $L_1$. This approach was integrated into the iterative support detection (ISD) method \cite{wang2010sparse}, which minimizes $\sum_{i\notin T}|x_i|$, where $T$ is a fixed set containing the indices of the large magnitude entries from the previous reconstruction.  Ma et al. \cite{ma2017truncated} developed the truncated $L_{t,1-2}$ model, which incorporates a similar idea into the $L_{1}$-$L_{2}$ model.
More recently, these truncated ideas have been incorporated into the $L_1/L_2$ framework:  \cite{l1linf} replaced the denominator into $L_\infty$ which is equal to truncate the entries to the one with the largest magnitude. \cite{ng_partial_l1dl2} adjust numerator into $L_{t,1}$ and prove  the restricted isometry property. 
In addition, Huang et al.~\cite{huang2015nonconvex} proposed a sorted $L_1$ minimization, where the truncated effect is transformed into a weighting scheme based on the ranks of the corresponding components among all the components in magnitude. 

In this study, we propose a novel nonconvex minimization model, called the sorted $L_1/L_2$, to promote sparsity in signal recovery. Our approach is inspired by the effectiveness of the $L_1/L_2$ method \cite{rahimi2019scale,tao2023study,wang2022minimizing,wang2019accelerated,wang2021limited} and the truncated/sorted models, such as the sorted $L_1$ method \cite{huang2015nonconvex}, which penalize components based on the order of their absolute value. Our model extends the sorted $L_1$ method by incorporating the $L_2$ norm and offers a flexible framework for promoting sparsity.
The major contributions are three-fold:
% \vspace{5mm
\begin{enumerate}[(a)]
    \item We propose the sorted $L_1/L_2$ model to solve the sparse signal recovery problem in noise-free and noisy scenarios. 
    \item We provide theoretical proof of the existence of the solution and employ a linearization algorithm for numerical optimization.
    \item We perform various experiments to showcase the effectiveness of our proposed method and demonstrate its superiority over current state-of-the-art techniques in both noise-free and noisy scenarios, with a particular focus on support detection.

%     \item We test the support detection to exhibit the
% superior behavior and also design an intriguing trial to effectively select the support, especially with our proposed model.
\end{enumerate}

% \vspace{5mm}
The rest of this paper is organized as follows. 
% We review some existing sparsity models and algorithms in \Cref{section: previous work}. 
The proposed minimization models are presented in detail in Session~\ref{sec:model} with a toy example. We also provide a theoretical analysis of the existence of solutions in both noisy and noise-free cases in Section \ref{section proof}. The algorithmic development for the two models is introduced in Section \ref{section: algorithms}. In Section \ref{section:numerical experiment}, we present the results of our numerical experiments, which include sparse recovery in various sensing matrices and models. Finally, we conclude with a summary of our findings and suggestions for future work in Section \ref{senction:conclusion}.

\section{Sorted $L_1/L_2$ minimization}\label{sec:model}
In this section, we introduce a novel nonconvex optimization model called the sorted $L_1/L_2$ minimization. Given the signal $\h x\in \mathbb{R}^n$, we can build a non-negative vector $\h w\in(0, 1]^n$ such that
$$
    0 < w_{i_1}\leq w_{i_2} \leq ...\leq w_{i_n}\leq 1,
$$
where the sequence is sorted by the magnitude of $\h x$ in a decreasing order $|x_{i_1}|\geq |x_{i_2}|\geq \dots \geq |x_{i_n}|$. The nonconvex sorted $L_1/L_2$ regularization is defined as 
\begin{equation}
    \label{eq:sorted_l1dl2}
    R_{\h w}(\h x):=\frac{\|\h w \odot \h x\|_1}{ \| \left({\bf{1}}-\h w \right) \odot\h x\|_2},
\end{equation}
where $\odot$ is  the Hadamard product or elemental-wise multiplication, i.e., $(\h p \odot \h q)_i:= p_iq_i$, $p_i \in \h p$, $q_i \in \h q$. Here we define $\frac{\|{{\bf 0}}\|_1}{\|{\bf 0}\|_2} = 0$ and require $\h w \neq \h 1$. 
% A worth-noticing different way to deal with the entries from the truncated-type model is that 

Sorted $L_1/L_2$ regularization shows high flexibility due to the weight vector applied on $\h x$. 
% We can establish the relationship between the sorted $L_1/L_2$ model with the $L_1/L_2$ model
When $w_i = \frac{1}{2}, \forall i$, \eqref{eq:sorted_l1dl2}
% , $0<c<1$, $c \in \mathbb{C}^1$, 
 becomes the original $L_1/L_2$ model. Besides, the scale-invariant property is preserved in the sorted $L_1/L_2$ minimization, which is exactly the same property as the $L_0$ model. 
In the numerator, 
we penalize the entries with small magnitudes by giving them a large weight such that the restored signal is encouraged to be sparse. On the other hand, we use a small weight to protect those entries with large magnitudes since these entries in the support should not be suppressed. In the denominator, we do the opposite way by assigning large weights for large magnitudes, and small weights for small magnitudes. This is because minimizing one over the $L_2$ norm is equivalent to maximizing the $L_2$ norm. In light of the above motivation, the big magnitude can be protected by assigning a large weight to a large magnitude for the denominator. 
% while we set a large weight for those entries with small magnitudes in order to detect more support indexes. In the denominator, we suppress those small ones into zero and only maintain the largest ones to shrinkage.
% Besides, the sorted way also urges to sort the recovered solution as the ground truth.
% \cw{I will add more insights here.}

With this sparse term, we get
a constrained method in the noise-free case
\begin{equation}\label{noise-free problem}
    \min\limits_{ \mathbf{x}\in \mathbb{R}^n} \ R_{\h w}(\h x) \text { s.t. } A\h x= \h b. 
\end{equation}
% \begin{equation}\label{noise-free problem}
%    \min _{\mathbf{x} \in \Rbb^{n}} \frac{\|\mathbf{w} \odot \mathbf{x}\|_{1}}{\|(\bf 1-\mathbf{w}) \odot \mathbf{x}\|_{2}}, \quad \text { s.t.} A \h x =\h b.
% \end{equation}
For the noisy case, an unconstrained minimization model can be formulated as
\begin{equation}\label{noisy problem}
      \min\limits_{ \mathbf{x}\in \mathbb{R}^n} \ \ \left\{\lambda  R_{\h w}(\h x)  + \frac{1}{2} \| A\h x- \h b\|_2^{2} \right\},
\end{equation}
% where $\h w$ is a nonnegative weight sequence bounded in $(0,1]$, and it is composed with $0 < w_{i_1}\leq w_{i_2} \leq ...\leq w_{i_n}\leq 1$, corresponding to the absolute values of $\h x$ ranked in decreasing order $|x_{i_1}|\geq |x_{i_2}|\geq \dots \geq |x_{i_n}|$, where $w_{i_j} \in \left\{w_i\right\}$, $x_{i_j} \in \left\{x_i\right\}$. 
where $\lambda$ is a regularization parameter.

\subsection{A toy example}\label{toy example}
Now we give a toy example to study the behaviors of the sorted model and compare it with different sparse models. Here, we consider a noise-free problem and aim to solve 
the sparsest solution that satisfies a given linear equation
% We will illustrate the good performance of sorted $L_1/L_2$ norm under certain appropriate situations and its drawback with the corresponding solutions.
$$
    \left[\begin{matrix}
1 & 0 & a & 0 \\
0 & 1 & -2 & 0 \\
0 & 1 & 0 & -2
\end{matrix}\right] \mathbf{x}=\left[\begin{array}{l}
1 \\
1 \\
1
\end{array}\right], 
$$
where $a\neq-2$ is a parameter. Notice that the linear equation is under-determined. Any general 
solution of this linear equation has the form of  $\h x = (-a k+1,2 k+1, k, k)$ for a scalar $k\in \mathbb{R}$. The  sparsest solution is $(1,1,0,0)$ obtained by $k=0$.
Here we plot the objective values of each mentioned vector for three particular values of $a$. Specifically, we follow \cite{huang2015nonconvex} and  set the weight vector $\h w$ in our proposed model as 
\begin{equation}
\label{eq:weight_vector}
    w_i= 
	\begin{cases}
	1	&	\text{otherwise}, 
	\\
	e^{-r(t-i)/t}	&	 i\in \Gamma_{\h x,t}.
	\end{cases}
\end{equation}
where $r$ is a ``slope rate'' constant controlling the curvature of the weight curve and $\Gamma_{\h x,t}$ is a set containing the index of the entries of
 $\h x$ with the $t$ largest magnitudes, i.e., for any $i \notin \Gamma_{\mathbf{x}, t}$, $j \in \Gamma_{\mathbf{x}, t},\left|x_{i}\right| \leq\left|x_{j}\right|$. 
 	\begin{figure*}[h]
		\begin{center}
			\begin{tabular}{ccc}
			   % (a) Toy example 1 $(a = -3.5) & Toy example 1 $(a = -3.5) \\
			   % $a=-3$  & $ a=3.5$ &  $a=4$ \\
				\includegraphics[height=0.25\textwidth]{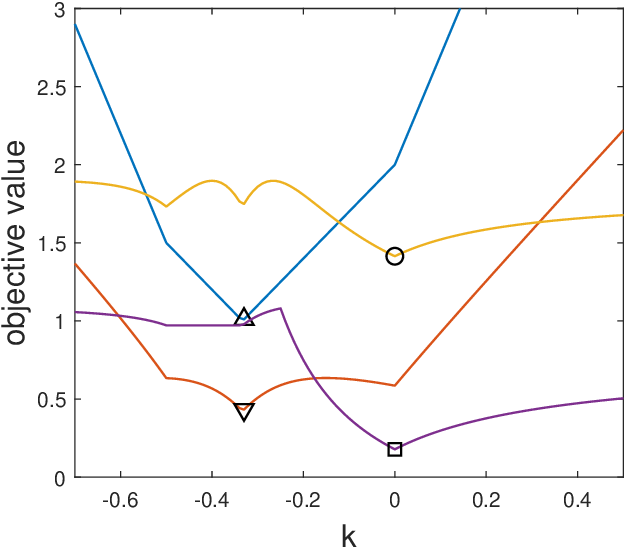} &
                    \includegraphics[height=0.25\textwidth]{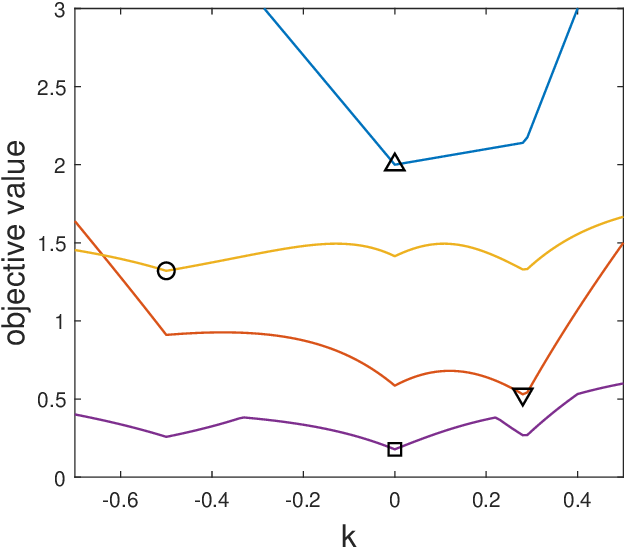}&
                   \includegraphics[height=0.25\textwidth]{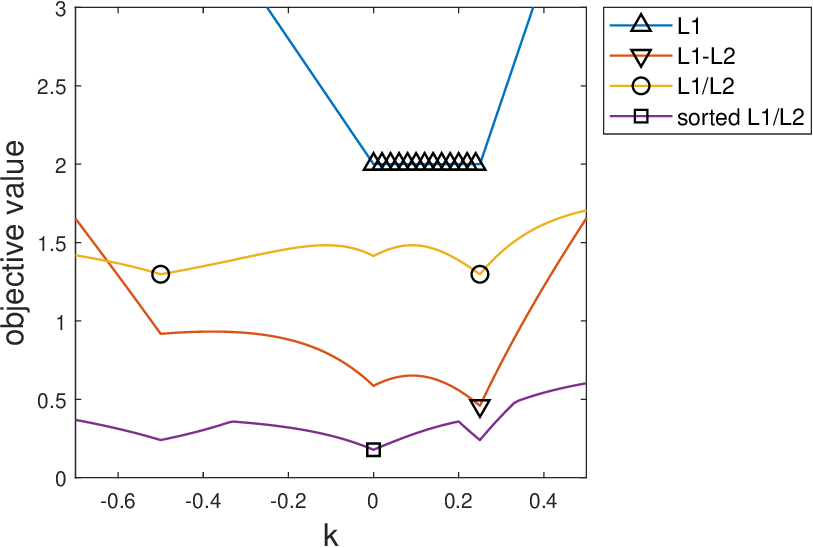} \\
			\end{tabular}
		\end{center}
		\caption{The objective value of different models when $a=-3$  (left), $ a=3.5$ (middle),  $a=4$ (right). Model parameters setting: sorted $L_1/L_2$: $t=2$, $r=5$.
		}\label{fig:toy1}
	\end{figure*}

In Figure~\ref{fig:toy1}, only our proposed sorted $L_1/L_2$ model can obtain the sparsest solution $k=0$ as the global minimizer for all the cases. The $L_1$ method fails when $a = -3$ while it has  multiple minimizers when $a=4$. The $L_1/L_2$ regularization only gets the sparsest solution when   $a = -3$, and $L_{1}$-$L_{2}$ fails for all cases.

\begin{figure*}[h]
		\begin{center}
			\begin{tabular}{ccc}
			   % (a) Toy example 1 $(a = -3.5) & Toy example 1 $(a = -3.5) \\
			   % $a=-3$  & $ a=3.5$ &  $a=4$ \\
				\includegraphics[height=0.25\textwidth]{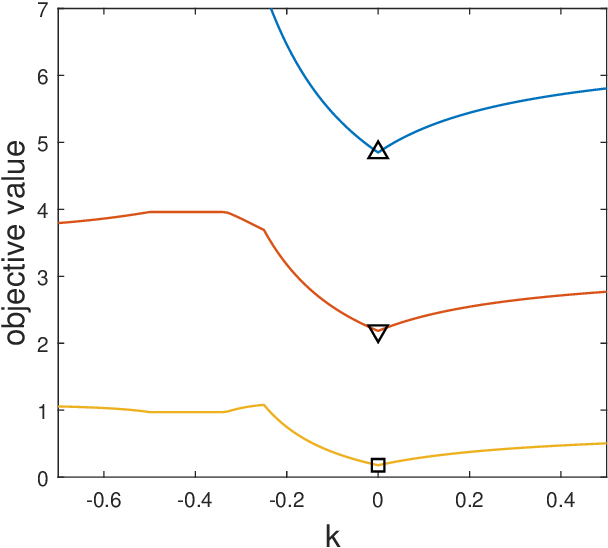} &
			    \includegraphics[height=0.25\textwidth]{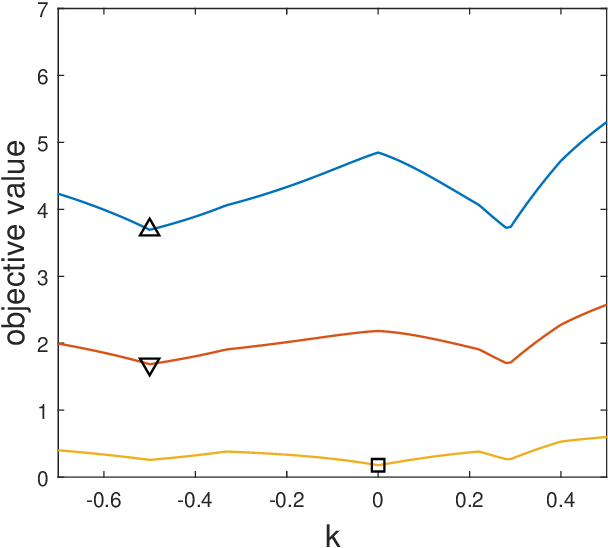} &
				\includegraphics[height=0.25\textwidth]{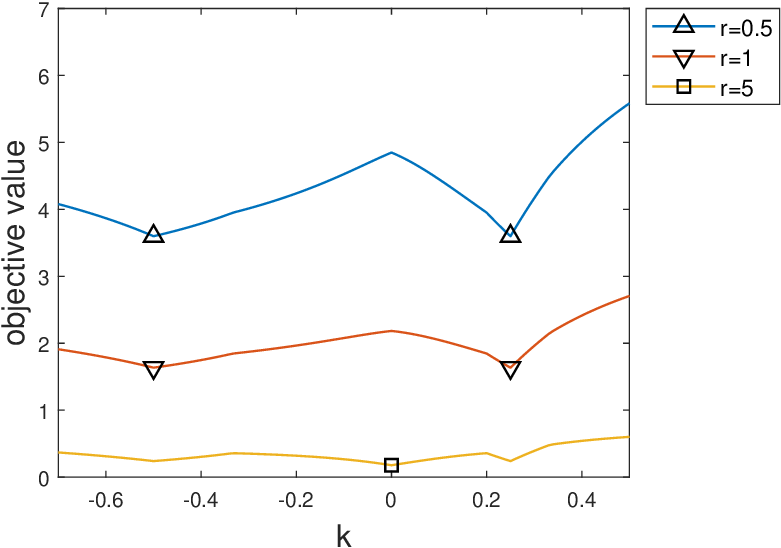} \\
			\end{tabular}
		\end{center}
		\caption{The effect of slope rate $r$ in the sorted $L_1/L_2$ model when $a=-3$  (left), $ a=3.5$ (middle),  $a=4$ (right).
		}\label{fig:toy2}
	\end{figure*}

Figure~\ref{fig:toy2} shows that $r$ affects the scale of the objective value.  A smaller value of $r$ corresponds to a larger weight thus it will result in a larger value of the numerator, as well as a smaller value of the denominator. 
% Therefore it generates a much larger scale of penalty, which reminds us to amplify the regularization parameter under certain cases.
In the case of $a=3.5$ and $a=4$, different values of $r$ lead to different optimal solutions. It is observed that when $r = 5$, the sorted model can find the sparsest solution in all the cases. 

\section{Theoretical analysis: Solution existence} \label{section proof}

 In this section, we will discuss the theoretical properties of the sorted $L_1/L_2$ model. 
 % First, we give several basic bounded results with the objective function and the recovered solution. 
 We adapt some analysis in  \cite{tao2022minimization,zeng2021analysis} to prove the solution existence of the sorted $L_1/L_2$ models in 
 % Since we proposed two types of primal problems in the noise-free and noisy case, 
% Our analysis is established
both the noise-free and noisy models. 
Note that $\{\h x^k\}$ is called a minimizing sequence of \eqref{noise-free problem}, if $A \h x^k=\h b $ for all $k$ and $\lim_{k\to \infty} R_{\h w} (\h x^k) = s^\ast$, where 
    \begin{equation}\label{eq:s^*}
            s^{*}=\inf _{\mathbf{x} \in \mathbb{R}^{n}}\left\{R_{\h w}(\h x) \mid A \mathbf{x}= \h b\right\}.
        \end{equation} 
Our analysis involves an auxiliary problem as follows
\begin{equation}\label{eq:p^*}
    p^* :=  \inf\limits_{\h x \in \mathbb{R}^n} \{ R_{\h w}(\h x) \mid \mathbf{x} \in \operatorname{ker}(A) \backslash\{\mathbf{0} \}\}.
\end{equation}

First, we show our proposed regularization $R_{\h w}(\h x)$ is bounded under a mild assumption. 
% Since \eqref{eq:sorted_l1dl2} is designed to be a fraction, $\bf{0}$ is the sparsest but trivial solution to the problem \eqref{noisy problem}.  
\begin{lemma}\label{prop:bounded penalty}
% Define $\h W = \mathrm{diag}(w_1, w_2, .., w_n)$. 
Given a nonzero signal $\h x\in \mathbb{R}^n$, if there exists  $ \delta>0$, such that $ w_i\geq \delta, \forall i$ then we have 
    \begin{equation}
        \delta \leq R_{\h w}(\h x)\leq \sqrt{\|\h x\|_0}\leq \sqrt{n}. 
        \label{bound of obj}
    \end{equation}
    \end{lemma}
\begin{proof}
% The lower bound can be obtained by the proof of Lemma~\ref{lemma: nonempty of p^*}. 
Here we can get a lower bound of $R_{\h w}(\h x)$ as 
$$
    % \begin{split}
        R_{\h w}(\h x)  \geq \frac{\|\h w \odot \h x\|_1}{\|\h x\|_2} 
         \geq \frac{\|\h w \odot \h x\|_2}{\|\h x\|_2}  
         \geq  \frac{\delta\|\h x \|_2}{\|\h x\|_2} = \delta.
    % \end{split}
$$
% The inequality can be obtained straightforwardly
% \begin{equation*}
%     % \begin{split}
%         R_{\h w}(\h x)  \geq \frac{\|\h w \odot \h x\|_1}{\|\h x\|_2} 
%          \geq \frac{\|\h w \odot \h x\|_2}{\|\h x\|_2}  
%          \geq  \frac{\delta\|\h x \|_2}{\|\h x\|_2} = \delta.
%     % \end{split}
% \end{equation*}
On the other hand, we have $\|\h x\|_1\leq \sqrt{\|\h x\|_0}\|\h x\|_2$ and $\h w$ is sorted by the magnitude of $\h x$ in a decreasing order, then 
$
     R_{\h w}(\h x) \leq \frac{\|\h x\|_1}{\|\h x\|_2} \leq \sqrt{\|\h x\|_0}\leq \sqrt{n}.
$
%     \begin{equation}
% \delta = \frac{\delta\|\h x \|_2}{\|\h x\|_2}\leq \frac{\|\h w \odot \h x\|_2}{\|\h x\|_2}\leq \frac{\|\h w \odot \h x\|_1 }{\|\left(\bf{1}-\h w \right) \odot\h x\|_2} \leq \frac{\|\h x_S\|_1}{(1-\xi) \|\h x_S\|_2} \leq \frac{\sqrt{n}}{1-\xi}.
% \end{equation}
\end{proof}

% This is reachable if choosing the values in the weight vector to be positive and lower than one.

Regarding the noise-free problem \eqref{noise-free problem}, it is clear to see that only when $\h b={\bf 0}$, we have $\h x=\bf 0$. Therefore,   the global optimal solution cannot be a trivial solution $\bf 0$, if ${\bf b} \neq{\bf 0}$. Next, we show the same conclusion in the noisy case  \eqref{noisy problem}. 
%% THM1 ----------------------------
\begin{theorem} 
Support $A$ is an under-determined matrix, $\h b \in \mathrm{Im}(A)$. Denote $F(\h x) := \lambda  R_{\h w}(\h x)  + \frac{1}{2} \| A\h x- \h b\|_2^{2}. $ 
% If $\h b \notin \ker(A)$, $\h b \in \Rbb^n$, there exists one subscript $j$ such that  $\h a _j^\top \h b \neq 0$ and $w_j\leq \frac{1}{2}$, 
For sufficiently small parameter $\lambda$, 
 the optimal solution of \eqref{noisy problem} cannot be ${\bf{0}}$. 
\label{thm:nonzero_noisy}
\end{theorem}
\begin{proof} 
Since $A$ is an under-determined matrix and $\h b \in \mathrm{Im}(A)$, there exist infinitely matrix solutions in $A \h x = \h b$. Denote 
\begin{equation*}
    \hat{\h x} = \arg\min_{\h x} \|\h x\|_2 \quad \text{ such that } \quad  A\h x = \h b. 
\end{equation*}
Based on Lemma~\ref{prop:bounded penalty}, we have 
\begin{equation*}
    F(\hat{\h x}) = \frac{\lambda\|\h w \odot \hat{\h x}\|_1}{ \| \left({\bf{1}}-\h w \right) \odot \hat{\h x}\|_2}  \leq \lambda \sqrt{\|\hat{\h x}\|_0},
\end{equation*}
Hence, if $\lambda <\frac{\| \h b\|_2^2}{2\sqrt{\|\hat{\h x}\|_0}}$, then we have $F(\hat{\h x})<  \frac{\|\h b\|_2^2
}{2}=F(\h 0)$, which leads
% Assume the sensing matrix $A = [\h a_1, \h a_2, ...,\h a_n] \in \Rbb^{m\times n}$, where each column $\h a_i \in \Rbb^m $. There exists
%  a constant $\gamma$ such that    
%     \begin{equation}
%     \label{gamma}
%     \gamma \in \left\{\begin{matrix}
% \left(0,\frac{2\h a_j^{\top}\h b}{\|\h a_j\|_2^2} \right ) & \h a_j^{\top}\h b > {\bf{0}} \\
% \left (\frac{2\h a_j^{\top}\h b}{\|\h a_j\|_2^2}, 0 \right) & \h a_j^{\top}\h b < {\bf{0}}  \\
% \end{matrix}\right. 
%     \end{equation}
% Hence we get 
% \begin{equation}
%     \begin{split}
%            &  \| \gamma \h a_j - \h b\|_2^2 = \|\h a_j \|_2^2 \gamma^2 -2\gamma\left< \h a_j, \h b\right> + \|\h b\|_2^2  \\ & = \gamma \left(\gamma\| \h a_j\|_2^2 - 2 \left< \h a_j, \h b \right > \right ) +\| \h b\|_2^2 < \| \h b \|_2^2.
%     \end{split}
% \end{equation}
% Following (\ref{gamma}), one can directly obtain
% a sparse vector with only one non-zero entry,
% $$
%     \widehat{\mathbf{x}}_{i}=\left\{\begin{array}{ll}
% \gamma, & i=j \\
% 0, & i \neq j, 
% \end{array}\right.
% $$
% such that 
% $$
%     \| A \widehat{\h x} - \h b\|_2^2 = \| \gamma \h a_j - \h b \|_2^2 < \|\h b \|_2^2.
% $$
% Since the corresponding $w_j\leq\frac{1}{2}$, it leads to
% \begin{equation}
%  \begin{split}
%         F(\widehat{\mathbf{x}}) & =\lambda  R_{\h w}(\widehat{\mathbf{x}})  + \frac{1}{2} \| A\widehat{\mathbf{x}}- \h b\|_2^{2} \\ & = \lambda \frac{w_j}{1-w_j}+\frac{1}{2}\|A \widehat{\mathbf{x}}-\mathbf{b}\|_{2}^{2} \\ & <\lambda+\frac{1}{2}\|\mathbf{b}\|_{2}^{2}  
%          =F(\mathbf{0}).
%  \end{split}
% \end{equation}
the global minimizer of (\ref{noisy problem}) can not be ${\bf 0}$. 
\end{proof}

Next we discuss the non-emptiness of the auxiliary problem \eqref{eq:p^*} and the relationship between $s^\ast$ and $p^\ast$. 

\begin{lemma} \label{lemma: nonempty of p^*}
    % Define $p^* :=  \inf\limits_{\h x \in \mathbb{R}^n} \{ R_{\h w}(\h x) \mid \mathbf{x} \in \operatorname{ker}(A) \backslash\{\mathbf{0} \}\},$
    % \left\{\frac{\|\h W \h z\|_1}{\|(\mathbf{I}-\h W)\h z\|_2} \right\}$ 
    % and a sequence ${\h z^k} $ such that $\h z^k \in \ker(A) \setminus \bf 0$, 
    Consider \eqref{eq:p^*},  assume $\h w$ is larger than zero, i.e., there exists  $ \delta>0$, such that $ w_i\geq \delta, \forall i$,   
 then   $p^*<+\infty$ and the solution of \eqref{eq:p^*} is nonempty for any $A\in \Rbb^{m\times n}$ and $\h x \in \Rbb^n$.
\end{lemma}
\begin{proof}
% Here we can get a lower bound of $R_{\h w}(\h x)$ as 
% $$
%     % \begin{split}
%         R_{\h w}(\h x)  \geq \frac{\|\h w \odot \h x\|_1}{\|\h x\|_2} 
%          \geq \frac{\|\h w \odot \h x\|_2}{\|\h x\|_2}  
%          \geq  \frac{\delta\|\h x \|_2}{\|\h x\|_2} = \delta.
%     % \end{split}
% $$
    % The infimum has a lower bound according to \Cref{prop:bounded penalty}. Here we could suppose 
    Suppose there exists an unbounded sequence $\{\h z^k\}$ of \eqref{eq:p^*} such that
      $ R_{\h w}(\h z^k)  \to p^*$ as $k\to +\infty$. It implies $p^*< +\infty$. Defining  
  $\hat{\h z}^k =\frac{ \h z^k}{\|\h z^k \|_2}$, which is bounded and belongs to $\operatorname{ker}(A) \backslash\{\mathbf{0} \}$. 
  Owing to the scale-invariant property of the sorted $L_1/L_2$,  we have $R_{\h w}(\hat{\h z}^k) = R_{\h w}(\h z^k) \to p^\ast$.Therefore, we get one accumulation point $\h z^\ast\in \operatorname{ker}(A) \backslash\{\mathbf{0} \} $ from this bounded sequence. 
  % thanks to \Cref{prop:Ax=0}, the sequence is also satisfying $\hat{\h z}^k \in \ker(A)$ and converging to the infimum of the penalty $p^*$. Since $\hat{\h z}^k$ is bounded due to \Cref{prop:bounded penalty}, thus it has an accumulation point such that $\Tilde{\h z} \in \ker(A)$. 
  Thus we prove the solution set of $p^*$ is nonempty.
\end{proof}

\begin{lemma} \label{s^*< p^*} Consider \eqref{eq:s^*} and \eqref{eq:p^*},
    for any $A\in \Rbb^{m\times n}$, $\h x \in \Rbb^n$, one can obtain
    $s^\ast \leq p^\ast$. 
    % \begin{equation}
    %     \inf _{\mathbf{x} \in \mathbb{R}^{n}}\left\{R_{\h w}(\h x) \mid A \mathbf{x}=\h b\right\} \leq \inf _{\mathbf{z} \in \mathbb{R}^{n}}\left\{R_{\h w}(\h z) \mid \mathbf{z} \in \operatorname{ker}(A) \backslash\{\mathbf{0}\}\right\} = p^*
    % \end{equation}
    \end{lemma}
    \begin{proof}
    %    Define \\
    % \begin{equation}\label{eq:s^*}
    %         s^{*}=\inf _{\mathbf{x} \in \mathbb{R}^{n}}\left\{R_{\h w}(\h x) \mid A \mathbf{x}= \h b\right\}.
    %     \end{equation} 
        For every $\h v \in \ker(A)\setminus \{{\bf 0}\}$, we will get $s^{*} \leq R_{\h w}(\h x+t \h v) $,
        % \begin{equation}
        %     s^{*} \leq R_{\h w}(\h x+t \h v) 
        %     % \frac{\|\mathbf{W(x+tv)} \|_{1}}{\|\mathbf{(I-W)x+(I-W)tv} \|_{2}},
        % \end{equation}
        where $t\in \Rbb$ is a constant.
        % , $\h v \in \Rbb^n$ is a disturbance vector. 
        Thus we derive
        \begin{equation} \label{lim}
        \lim _{t \to \infty} R_{\h w}(\h x+t \h v) =\lim _{t \rightarrow \infty}R_{\h w}(\h x/t+ \h v) = R_{\h w}( \h v).
            % \lim _{t \to \infty} \frac{\|\mathbf{W(x+tv)}\|_{1}}{\|\mathbf{(I-W)(x+tv)}\|_{2}}=\lim _{t \rightarrow \infty} \frac{\|\mathbf{Wx}/ t+\mathbf{Wv}\|_{1}}{\|\mathbf{(I-W)x}/ t+\mathbf{(I-W)v}\|_{2}}=\frac{\|\mathbf{Wv}\|_{1}}{\|\mathbf{(I-W)v}\|_{2}}.
        \end{equation}
        Eventually, we will obtain
        \begin{equation}\label{eq:s^*< lim}
           s^{*}  \leq  R_{\h w}( \h v),
        \end{equation} 
        where $\h v \in \Rbb^n$, $\h v \in \ker(A)\setminus \{{\bf 0}\} $, which leads to $s^\ast \leq p^\ast$, directly. 
    \end{proof}

% \begin{lemma}
%    Define $s^* = \lim _{t \rightarrow \infty} \frac{\|\mathbf{Wx^t}\|_1}{\|\mathbf{(I-W)x}^t\|_2}$, then $s^*=p^*$ if and only if there exists a minimizing sequence of (\ref{noise-free problem}) that is bounded.
% \end{lemma}
% \begin{proof}
%     Suppose there is an unbounded minimizing sequence $\left\{\h x^t \right\}$ of (\ref{noise-free problem}). Define $\h x^*$ the same way as \ref{prop:Ax=0}. Without loss of generality, assume we have $\|\h x^t\| \to \infty$ and  thus one can obtain $ \delta \leq \|\h x^*\|_2 \leq \frac{\sqrt{n}}{1-\xi}$ and $s^* = \|\h x^*\|_1$.  
%     After that, 
       
%         The inequality implies $A \h x^* = \bf 0$. 
%         One can get 
%         \begin{align}
%             p^* =  \inf_ {\h z \in \Rbb^n} \left (\frac{\|\h W \h x\|_1}{\|(\mathbf{I}-\h W)\h x\|_2} \right) &\leq \frac{\| \mathbf{Wx^*}\|_1}{\|\mathbf{(I-W)x^* \|_2}} \\ 
%             & \leq \frac{\|\mathbf{x^*} \|_1}{(1-\xi)\|\mathbf{x^*} \|_2}\\
%             & \leq \frac{1}{\delta(1-\xi)} \|\mathbf{\h x^*}\|_1 \\
%            % & \leq \| \h x^* \|_1 \\
%            % & = s^* \\
%            % & < \infty.
%         \end{align}
%   since $\delta$, $\xi$ satisfy $0 < \delta < \xi <1$, thus clearly one can obtain $\frac{1}{\delta (1-\xi)} > 1$.

% \end{proof}
By the definition in \eqref{eq:s^*}, we suppose there exists an unbounded minimizing sequence $\h x^k$ converge to $s^*$, i.e., $\lim_{k \to \infty} R_{\h w}(\h x^k)= s^*$. After that, we now prepare to prove that $p^*$ is equivalent to $s^*$. 
\begin{lemma}
\label{thm:s^*=p^*} Consider \eqref{eq:s^*} and \eqref{eq:p^*},  for any $A\in \Rbb^{m\times n}$, $\h x \in \Rbb^n$, one can obtain 
     $s^*=p^*$ if and only if there exists a minimizing sequence of \eqref{eq:s^*} that is unbounded. 
     % The solution set of \eqref{noise-free problem} is nonempty.
\end{lemma}
\begin{proof}
    Suppose there is an unbounded minimizing sequence $\left\{\h x^k \right\}$ of the noise-free problem (\ref{eq:s^*}). Without loss of generality, assume we have $\|\h x^k\|_2 \to \infty$ and denote $ \lim_{k \to \infty} \frac{\h x^k}{ \| \left({\bf{1}}-\h w \right) \odot\h x^k\|_2}= \h x^\ast$ for some $\h x^\ast$ where  $\|\h w \odot \h x^\ast\|_1 = s^\ast$ according to the definition of the minimizing sequence. In addition, 
\begin{equation}
    % \begin{split}
                A \h x^\ast = \lim_{k \to \infty} \frac{A \h x^k}{ \| \left({\bf{1}}-\h w \right) \odot\h x^k\|_2} = \lim_{k \to \infty} \frac{\h b}{ \| \left({\bf{1}}-\h w \right) \odot\h x^k\|_2}  = {\bf 0}. 
    % \end{split}
\end{equation}
    Hence $\h x^\ast \in \ker(A)\setminus \{{\bf 0}\} $. One can obtain $\|\mathbf{(1-w)}\odot \mathbf{x}^* \|_2 =1$, which implies
    % $p^*\leq R_{\h w}(\h x^\ast)$. In addition, we have 
    % One can obtain $ 1 \leq \|\h x^*\|_2 \leq \frac{1}{1-\xi}$ and $s^* = \| {W \h x^*}\|_1$.   After that, one can get 
    %     \begin{align}
    %         p^* =  \inf_ {\h z \in \Rbb^n} \left\{R_{\h w}(\h z) \mid \h z\in \ker(A)\setminus \{{\bf 0}\}  \right\} &\leq R_{\h w}(\h x^\ast),
    %     \end{align}
% notice $\| {(I-W) \h x}^* \|_2 = \lim_{t \to \infty} \frac{\|{(I-W) \h x^t}\|_2}{\| {(I-W) \h x^t} \|_2} = 1$, thus clearly we can get

\begin{equation}
    p^*\leq R_{\h w}(\h x^\ast) = \frac{\|\h w \odot \h x^\ast\|_1}{\|\mathbf{(1-w)}\odot \mathbf{x}^* \|_2} = s^*. 
\end{equation}
% and thanks to \Cref{prop:bounded penalty}, $s^*$ has been upper and lower bounded. Therefore $p^*\leq s^* < \infty$. 
Combining Lemma~\ref{s^*< p^*},  we obtain $p^* = s^*$. 
% Then we fix any $\h x$ such that $A\h x = \h b$ and suppose $\h v \in \ker(A)\setminus{\bf0} $, and it holds that 
% \begin{equation}
%      s^* \leq \frac{\| {W \h v}\|_{1}}{\|{(I-W) \h v}\|_{2}} ,  
% \end{equation}
% thanks to \eqref{eq:s^*< lim}. 
% \jj{Start}

Now we prove the necessary condition:
suppose that $s^* = p^*$,  since $s^*<\infty$, we have  a sequence $\left\{\mathbf{d}^k\right\}$ satisfying $\mathbf{d}^k \in\ker(A)\setminus \{{\mathbf{0}}\}$ such that $\lim_{k \to \infty}R_{\h w}(\h d^k)= p^*$. Without loss of generality,
assume $\lim_{k \to \infty} \frac{\mathbf{d}^k}{\|\mathbf{(1-w)}\odot{\mathbf{d}^k}\|_2} = \mathbf{d}^*$ for some $\mathbf{d}^*$ with $\|{\mathbf{(1-w)}}\odot {\mathbf{d}^*}\|_2 =1$. It follows that 
\begin{equation}
\begin{split}
       &  A \mathbf{d}^* = \lim_{k \to \infty} \frac{A\mathbf{d}^k}{\|{\mathbf{(1-w)}}\odot {\mathbf{d}^k}\|_2} = \mathbf{0} \\ & \text{with}\quad \|\mathbf{w}\odot\mathbf{d}^*\|_1 = \lim _{k \rightarrow \infty} R_{\h w}(\h d^k)= p^*.
\end{split}
\end{equation}
Now we choose one fixed $\mathbf{\bar{x}}$ such that $A\mathbf{\bar{x}}=\mathbf{b}$ and define ${\mathbf{x}^l}=\mathbf{\bar{x}}+l{\mathbf{d}^*}$ for $l = 1,2,...$. One can obtain $A{\mathbf{x}}^l=\mathbf{b}$ for all $l$. Thanks to \eqref{lim}, then the following equality achieved
\begin{equation}
       \lim_{l \to \infty}R_{\h w}(\h x^l) = 
      \lim_{l \to \infty}R_{\h w}(\bar{\h x}+l\h d^*) = 
       R_{\h w}(\h d^*)=p^\ast=s^*\label{p^*=s^*}
\end{equation}
since $\|\mathbf{x}^l \|_2 \to \infty$ as $l\to\infty$.
Thus, ${\mathbf{x}^l}$ is an unbounded minimizing sequence for \eqref{noise-free algorithm}, which indicates our proof is complete.
\end{proof}
%
%
% \jj{Start}
% Notice $R_{\h w}(\h d^*) = s^*$ in \eqref{p^*=s^*}, which implies the sequence $\h d^*$ makes the objective value obtain its minimum by the definition of $s^*$, i.e., $\h d^*$ is the optimal solution of \eqref{noise-free problem}.  
%  \jj{End}

 In \cite{tao2023study,zeng2021analysis,vavasis2009derivation,zhang2013theory}, the existence of globally optimal solutions is analyzed based on the spherical
section property (SSP). Here we revisit SSP and prove the solution's existence. 

% \jj{Start}
\begin{definition}(\text{spherical section property} \cite{vavasis2009derivation,zhang2013theory})
\label{eq:sSSP}
% Spherical section property.
Let $m$, $n$ be two positive integers such that $m<n$ and  $V$ be an $(n-m)$-dimensional subspace of $\mathbb R^n$. We say that $V$ has the $s$-spherical section property if
\begin{equation}
    % \inf _{v \in V \backslash\{0\}} \frac{\|\h w \odot \h v\|_{1}}{\|\h{(1-w)} \odot \h v\|_2} \geq \frac{1}{1-\xi}\sqrt{\frac{m}{s}}.
    \inf _{\h v \in V \backslash\{\bf 0\}}\frac{\|\h v\|_1}{\|\h v\|_2} \geq \sqrt{\frac{m}{s}}.
\end{equation}
\end{definition}

% \begin{remark}
% If  a subspace $V$ satisfies $s$-spherical section property, then there is a corresponding $\hat{s}:=\frac{s}{(\delta-\delta\xi)^2}$ such that $V$ has $\hat{s}$-sorted spherical section property. This is because: 
% \begin{equation}
%     \inf _{\h v \in V \backslash\{\bf 0\}}\frac{\|\h v\|_1}{\|\h v\|_2} \geq \inf _{\h v \in V \backslash\{\bf 0\}} \delta\frac{\|\h v\|_1}{\|\h v\|_2} \geq \delta \sqrt{\frac{m}{s}} =  \frac{1}{1-\xi}\sqrt{\frac{m}{\hat{s}}}.
% \end{equation}
% Hence, we can follow  the previous work \cite{zhang2013theory} and obtain the nullspace of the standard Gaussian 
% % has a high probability to  if $\delta\leq w_i \leq \xi<1$, $A \in \mathbb{R}^{m\times n} \text{ with } m<n$ is a random matrix with i.i.d. standard Gaussian entries, then its $(n-m)$-dimensional nullspace
% has the sorted spherical section property for $s = \frac{c_1}{((1-\xi)\delta)^2}(1+\log\frac{n}{m})$ with probability at least $1-e^{-c_{0}(n-m)}$, where $c_0$ and $c_1$ are positive constants independent of $m$ and $n$. 
% \end{remark}
\begin{theorem}  (solution existence in \eqref{noise-free problem})
Consider \eqref{noise-free problem}, suppose that $\ker(A)$ has the sorted spherical section property for some $s>0$ and there exists $\tilde{\h x}\in\mathbb{R}^n$ such that $\|\tilde{\h x}\|_0 < \delta^2 m/s$ and $A \tilde{\h x} = \h b $. If there exists  $ \delta>0$, such that $ w_i\geq \delta, \forall i$, then the set of optimal solutions of \eqref{noise-free problem} is nonempty.
\end{theorem}
\begin{proof}
    According to the sorted spherical property of $\ker (A)$ and the definition of $p^*$ in 
\eqref{eq:p^*}, one can obtain $p^*$ satisfies  
$p^* \geq \delta\sqrt{m/s}$. According to Lemma~\ref{prop:bounded penalty}, it follows that
\begin{equation}
        s^*  \leq \frac{\|\h w \odot \tilde{\h x}\|_1}{\|\h {(1-w)} \odot \tilde{\h x}\|_2} \leq \sqrt{\|\tilde{\h x}\|_0}
    < \delta\sqrt{\frac{m}{s}} \leq p^*,
\end{equation}
where the left term is satisfied with $A \tilde{\h x}=\h b$, and the middle term is due to the equivalence of norms, while the last term holds by our assumption. Thus $s^\ast <p^\ast$. Combining Lemma~\ref{thm:s^*=p^*}, we obtain that there is a bounded minimizing sequence $\left\{\h x^k \right\}$ for 
\eqref{noise-free problem}. One can pass to a convergent subsequence $\left\{\h x^{k_j}\right\}$ so that $\lim_{j \to \infty}x^{k_j} = \bar{\h x}$ for some $\bar{\h x}$ satisfying 
$A \bar{\h x} = \h b$. Since $\h b \neq \bf 0$, it implies $\bar{\h x}\neq \bf 0$. We then upon using the continuity of $R_{\h w}(\h x)$ at $\bar{\h x}$ and the definition of minimizing sequence that 
\begin{equation}
    R_{\h w}(\bar{\h x})=\frac{\|\h w \odot \bar{\h x}\|_1}{\|(\h {1-w}) \odot \bar{\h x}\|_2}
    = \lim_{j\to\infty}\frac{\|\h w \odot \bar{\h x}^{k_j}\|_1}{\|(\h{1-w})\odot \bar{\h x}^{k_j}\|_2} = s^*.
\end{equation}
This shows that $\bar{\h x}$ is an optimal solution of \eqref{noise-free problem}. Then the proof is completed.

\end{proof}

% \jj{End}

% Thus we obtain $p^* = s^*$ by combining \Cref{s^*< p^*}. From \Cref{lemma: nonempty of p^*}, then we, at last, get that the solution of \eqref{noise-free problem} is nonempty.
% \begin{equation}
%      s^* \leq \frac{\|\h W(\h x+s \h v) \|_1}{\|(\mathbf{I}-\h W)(\h x + s \h v)\|_2}, 
% \end{equation}
% by the definition of $s^*$ for any $s\in \Rbb$. Then it follows that 
% \begin{equation}
%     s^* \leq \lim_{s \to \infty}\frac{\|\h W(\h x+s \h v) \|_1}{\|(\mathbf{I}-\h W)(\h x + s \h v)\|_2} = \frac{\|\mathbf{Wv}\|_{1}}{\|\mathbf{(I-W)v}\|_{2}} ,   

% \end{proof}

% \begin{theorem} (Solution existence in \eqref{noise-free problem})
%     Consider \eqref{eq:p^*} and \eqref{eq:s^*}, if $s^\ast<p^\ast$, then the solution of \eqref{noise-free problem} is nonempty. 
% \end{theorem}
% \begin{proof}
%   It follows directly from \Cref{thm:s^*=p^*}:  $s^\ast<p^\ast$ shows that there exists a bounded minimizing sequence of \eqref{noisy problem}. Then the corresponding  accumulation point is the  minimizer.  
% \end{proof}
 
% Define a sequence $ \left\{\h x^k \right\}$ such that
% $\lim_{k \to \infty} F(\h x^k) = F^*$. 
%%----THM2
\begin{theorem} (solution existence in \eqref{noisy problem}) Denote the infimum of the objective function as 
\begin{equation}
\label{F*<infty}
    F^*:= \inf_ {\h x \in \Rbb^n} \lambda R_{\h w}(\h x)+\frac{1}{2}\|A \h x - \h b\|_2^2  < \infty. 
\end{equation} If $A\in \Rbb^n$ is full row rank,  $\h b \notin {\bf 0}$, there exists  $ \delta>0$, such that $ w_i\geq \delta, \forall i$, and $0 < \lambda < \frac{\|\h b\|_2^2}{2(\sqrt{n}-\delta )}$, then the solution set of (\ref{noisy problem}) is nonempty.
\end{theorem}
\begin{proof} 
Since $A$ is full row rank, one can obtain $\Tilde{\h x}:= A^{\top}( A A^{\top})^{-1}  \h b$ such that $A\Tilde{\h x} = \h b$. According to Theorem~\ref{thm:nonzero_noisy}, there exist a sequence taking from $\{\h x^k\}$ where $ \h x^k\neq {\bf 0}$, such that $\displaystyle \lim_{ k\to \infty}F(\h x^k) =  F^* $. 
% Then we can prove that considering  the sequence ${\h x^k}$ as the accumulation point, 
Then we will show the optimal solution being nonempty via the following two cases: 

\noindent
(\rmnum{1}) Assuming that there is only a finite number of $k$ such that $\h x^k \notin \ker(A)$, without loss of generality, if erasing these elements and maintaining the rest still as $\h x^k$, it will make the following inequality set up,
\begin{equation}
      \lambda \sqrt{n} \geq F(\tilde{\mathbf{x}}) \geq F^{*}=\lim _{k \rightarrow \infty} F\left(\mathbf{x}^{k}\right) \geq {\lambda}{\delta}+\frac{1}{2}\|\mathbf{b}\|_{2}^{2}, 
\end{equation}
% \jj{here should be :}
which violates the condition $0 < \lambda < \frac{\|\h b\|_2^2}{2(\sqrt{n}-\delta )}$.\\

% \jj{end}
% which violates the condition $\h 0 < \lambda < \frac{\delta \|\mathbf{b}\|_{2}^{2}}{2(\sqrt{n}-1)} $.\\
\noindent
(\rmnum{2}) If there are  infinite number of $k$ such that $\h x^k \notin \ker(A)$. Thus there must be a subsequence $\left\{ \h x^{k_i}\right\}$ such that $\left\{ \h x^k\right\}\supseteq  \left\{ \h x^{k_i} \right\} $ with $\h x^{k_i} \notin \ker(A)$. We can obtain
\begin{equation}
\begin{split}
        & \sqrt{\sigma_{\min }\left(A^{\top} A\right)}\left\|\mathbf{x}^{k_{j}}\right\|_{2}-\|\mathbf{b}\|_{2}  \leq\left\|A \mathbf{x}^{k_{j}}\right\|_{2}-\|\mathbf{b}\|_{2}  \\ & \leq\left\|A \mathbf{x}^{k_{j}}-\mathbf{b}\right\|_{2} \leq 2 \sqrt{F^{*}}.
\end{split}
\end{equation}
Here the first term is the smallest eigenvalue of the matrix $A^\top A$. The last term is due to $\|A \h x^{k_i} -\h b\|_2^2 \leq 2F(\h x^{k_i}) \leq 4F^*$ with a sufficient large subscript $i$. The result implies $\{\h x^{k_i}\}$ is bounded, thus there exists a subsequence ${\h x^{k_{i_j}}}$ converging to $\hat{\h x}$ such that,
\begin{equation}
    \lim _{j \rightarrow \infty} F(\h x^{k_{i_j}}) = F^*.
\end{equation}
Thus $\hat{\h x} \in \mathop{\arg\min}\limits_{\h x\in \mathbb{R}^n} F(\h x)$, which is a global minimizer of \eqref{noisy problem}, and hence the solution set is nonempty.
\end{proof}

%%------THM2 END

\section{Algorithms development} 
\label{section: algorithms}
%In this session, we consider the sorted $L_1/L_2$ models applied for both the noise-free and noisy cases and provide the corresponding algorithms, respectively.  
This section provides algorithms for the noise-free~\eqref{noise-free problem} and noisy models~\eqref{noisy problem}.

\subsection{Noise-free case}\label{section:algorithm1}

%  The performance of sorted $L_1/L_2$ relies on the quality of the initial guess.
 The constrained model \eqref{noise-free problem} can be rewritten as
\begin{equation}\label{equ:sorted}
    \min\limits_{ \mathbf{x}\in \mathbb{R}^n} \ \ \left\{  R_{\h w}(\h x) +  I_0(A\h x- \h b) \right\},
\end{equation}
where $I_0$ is an indicator function enforcing $\h x$ to satisfy the constraint
	\begin{equation}\label{eq:indicator}
	I_0(\h t) =
	\begin{cases}
	0	&	 \mbox{if } \h t ={\bf 0},
	\\
	+\infty	&	\text{otherwise}.
	\end{cases}
	\end{equation}
% $\lambda$ is the regularization parameter, $\h w$ is the weight and $\odot$ is called the Hadamard product or elemental-wise multiplication, i.e., $(\h p \odot \h q)_i:= p_iq_i$, $p_i \in \h p$, $q_i \in \h q$. 
% The minimization problem \eqref{equ:sorted} will become \eqref{equ:truncted} if $\h w$ is set as 
% \begin{equation*}
% 	w_i= 
% 	\begin{cases}
% 	1	&	 i\notin \Gamma_{x,t}
% 	\\
% 	0	&	\text{otherwise}.
% 	\end{cases}
% 	\end{equation*}
% Motivated by the above thought, to soft the ''hard-threshold'' in the truncated-way model, 
% we consider a weight $0 <w_i \leq 1$ to change gently rather than erase the entries. To be more precise, the weight function is given by
% \begin{equation*}
% 	w_i= 
% 	\begin{cases}
% 	1	&	 i\notin \Gamma_{x,t}, 
% 	\\
% 	e^{-r(t-i)/t}	&	\text{otherwise}.
% 	\end{cases}
% \end{equation*}
% As illustrated in section \ref{section: previous work}, $r$ is the slope rate to control the curvature of the function when $i \in \Gamma_{x,t}$. $t$ is the truncated number, i.e. the cardinality of $\Gamma_{x,t}$. $w_i, i\in \Gamma_{x,t} $ is an exponential family which gives a  sharp fitting to the $\Gamma_{xg,t}$, and is also for distinguishing the largest magnitude entries, where $xg$ denotes the ground truth. 
Here we adopt a DCA-type algorithm to solve the problem \eqref{equ:sorted}.   
Note that DCA is a descent algorithm for minimizing the difference of convex (DC) functions~\cite{TA98,phamLe2005dc}. Here, we use the same scheme but relax the limitation of convex functions. Generally, DCA splits the objective function into two terms $G$ and $H$: 
$$
\min _{\mathbf{x} \in \mathbb{R}^{m}} F(\mathbf{x}):=G(\mathbf{x})-H(\mathbf{x}). 
$$
Starting from an initial point $\mathbf{x}^{0}$,  DCA iteratively constructs two sequences $\left\{\mathbf{x}^{k}\right\}$ and $\left\{\mathbf{y}^{k}\right\}$:
$$\left\{\begin{array}{l}
\mathbf{y}^{k} \in \partial H\left(\mathbf{x}^{k}\right) \\
\mathbf{x}^{k+1}\in \argmin\limits_{\mathbf{x}} G(\mathbf{x})-\left\langle \mathbf{x}, \mathbf{y}^{k}\right\rangle, 
\end{array}\right.$$
where $\mathbf{y}^{k} \in \partial H\left(\mathbf{x}^{k}\right)$, i.e., $ \mathbf{y}^{k}$  is a subgradient of $ H(\mathbf{x}) \text { at } \mathbf{x}^{k}. $
In our  problem \eqref{equ:sorted}, we consider the decomposition as follows: 
\begin{equation}
    \begin{split}
        G(\h x) &= \alpha \|\h x\|_1 +  I_0(A\h x- \h b),\\
        H(\h x) & = \alpha \|\h x\|_1-R_{\h w}(\h x).
    \end{split}
\end{equation}
Here $G(x)$ is convex, but $H(x)$ may may not be.
Note that the selection of the weight guarantee $\|\mathbf{(1- w)} \odot\h x\|_2$ does not equal zero, and we get 
\begin{equation}
\label{eq:partialH_noisefree}
    \begin{split}
        \partial H(\h x) &= \alpha \partial \|\h x\|_1-\frac{\partial \|\h w \odot \h x\|_1}{ \|\left({\bf{1}}-\h w \right)\odot\h x\|_2}+\frac{ \|\h w \odot \h x\|_1 (1-\h w)^2\odot \h x}{ \|\left({\bf{1}}-\h w \right)\odot\h x\|_2^3}.
    \end{split}
\end{equation}
Therefore, $\h y^k$ can be set as  $\h y^k =\alpha\operatorname{sign}(\h x^k)-\frac{\operatorname{sign}(\h w \odot \h x^k)}{ \|\left({\bf{1}}-\h w \right)\odot\h x^k\|_2}+\frac{ \|\h w \odot \h x^k\|_1 (1-\h w)^2\odot \h x^k}{ \|\left({\bf{1}}-\h w \right)\odot\h x^k\|_2^3}, $ where $\operatorname{sign}(\h v)$ is a signum function
$
    	\operatorname{sign}(\h v)_i= 
	\begin{cases}
	1	&	 v_i>0, 
	\\
	0	&	 v_i = 0, 
	\\
	-1	&	v_i<0.
	\end{cases}
$
After computing $\h y^k$, the $\h x$-subproblem can be formulated as a linear programming problem. Notice that
\begin{equation}
\label{eq:update_x_noisefree}
    \mathbf{x}^{k+1}=\argmin _{\mathbf{x}} \left\{\alpha \|\mathbf{x} \|_1 + I_0 (A \h x -\h b)-\left\langle\mathbf{x}, \mathbf{y}^{k}\right\rangle\right\},
\end{equation}
which can be formulated into a linear programming (LP) problem by nonnegative transformation with $\h x$.
Assume $\h x \in \Rbb^n$ can be split into the minus of two nonnegative parts $\h x = \h x^+ - \h x^-$, where $\h x^+ \geq {\bf 0}$ and $\h x^- \geq {\bf 0}$
and let $\hat{\h x} = \begin{bmatrix} \h x^+\\\h x^-\end{bmatrix} $. Then the linear constraint $A \h x = \h b$ turns to $\Bar{A} \hat{\h x} = \h b$ where $\Bar{A} = \left [ A -A \right ]$. Then the problem will be 
$$
    \min _{\mathbf{x} \geq 0} \mathbf{1}^T \mathbf{x} \quad \text { s.t. } \quad \bar{A} \mathbf{x}=\mathbf{b},
$$
thus after deriving the solution $\hat{\h x}$, we could obtain the solution $\h x$ by
$$
\mathbf{x}=\hat{\mathbf{x}}(1: n)-\hat{\mathbf{x}}(n+1: 2 n).
$$
Besides, since we have the indicator function in $G(\h x)$, we could formulate it to
\begin{equation}
    \begin{aligned}
\mathbf{x}^{k+1} 
& =\argmin _{\mathbf{x}}\left\{\alpha\|\mathbf{x}\|_{1}+I_0(A \mathbf{x} - \h b)-\left\langle x, \partial H\left(\mathbf{x}^{k}\right)\right\rangle\right\} \\
& =\argmin _{\mathbf{x}}\left\{\alpha\|\mathbf{x}\|_{1}-\left\langle\mathbf{x}, \h y^k \right\rangle \quad \text {s.t. } A \mathbf{x}=\mathbf{b}\right\},
\end{aligned}
\end{equation}
which can be efficiently solved by an optimization software called Gurobi.

In addition, since we do not assume to know the order of the magnitude in the ground truth signal $\bar{\h x}$. We sort the entries of  $\h x^k$ during each iteration and update the weight accordingly. 
The algorithm is summarized as Algorithm~\ref{noise-free algorithm}.

\begin{algorithm}
     \caption{The sorted  $L_1/L_2$ minimization  via DCA-type algorithm in the noise-free case.}
     \begin{algorithmic}[0]
         \STATE{{\bf Input:} $A\in \mathbb{R}^{m\times n}, \h b\in \mathbb{R}^{m\times 1}$, Max and $ \epsilon \in \mathbb{R}$, $\rho\in R$ }
         % \STATE{Compute the $L_1$ model and obtain the restored solution $\h x_{L_1}$. }
         \STATE{{\bf Initialization: }   $k=1$ and solve the $L_1$ minimization to get $\h x^1$}
       % Set initial guess as $\h x_{L_1}$ and set $\h w$ according to  $\h x_{L_1}$. }
         \WHILE{$k < $ Max or $\|\h x^{k}-\h x^{k-1}\|_2/\|\h x^{k}\| > \epsilon$}
         \STATE{$\h y^k =\alpha\operatorname{sign}(\h x^k)-\frac{\operatorname{sign}(\h w \odot \h x^k)}{ \|\h x^k-\h w\odot\h x^k\|_2}+\frac{ \|\h w \odot \h x^k\|_1 (1-\h w)^2\odot \h x^k}{ \|\h x^k-\h w\odot\h x^k\|_2^3}$}
         \STATE{Adopt Gurobi to solve the linear programming problem and obtain $\h x^k$}
         \STATE{Update the weight $\h w$ based on $\h x^k$}
         \STATE{$k = k+1$}
         \ENDWHILE
         \RETURN $\h x^{k}$ \end{algorithmic} \label{noise-free algorithm} 
\end{algorithm}

% Here, we prove the proposed algorithm is a descent method. 
% \begin{theorem}
% \label{thm:descent_noisefree}
%    The sequence $\{\h x^k\}$ produced by Algorithm~\ref{noise-free algorithm} satisfies
% \begin{equation}
% \label{eq:descent_noisefree}
%        R_{\h w}(\h x^{k+1})\leq R_{\h w}(\h x^k). 
% \end{equation}
% \end{theorem}
% \begin{proof}
%     Based on the update scheme in \eqref{eq:update_x_noisefree}, we have 
%     \begin{equation}
%     \label{eq:th1}
%         \alpha \|\mathbf{x}^{k+1} \|_1 -\left\langle\mathbf{x}^{k+1}, \mathbf{y}^{k}\right\rangle \leq \alpha \|\mathbf{x}^{k} \|_1 -\left\langle\mathbf{x}^{k}, \mathbf{y}^{k}\right\rangle.
%     \end{equation}
%     On the other hand, $\h y^k\in \partial H(\h x^k)$, which leads 
%     \begin{equation}
%     \label{eq:th2}
%           \alpha \|\mathbf{x}^{k} \|_1 - R_{\h w} (\mathbf{x}^{k}) + \left\langle\mathbf{x}^{k+1}-\mathbf{x}^{k}, \mathbf{y}^{k}\right\rangle \leq \alpha \|\mathbf{x}^{k+1} \|_1 - R_{\h w} (\mathbf{x}^{k+1}). 
%     \end{equation}
%     Combining \eqref{eq:th1} and \eqref{eq:th2}, we prove \eqref{eq:descent_noisefree}.
% \end{proof}
\subsection{Noisy case}\label{section:algorithm2}

We consider the unconstrained model \eqref{noisy problem}. 
% Considering a signal $\h x\in \Rbb^n$, these entries obey the Gaussian distribution with zero-mean and standard variance $\sigma=0.1$. The matrix $\mathbf{A}$ is the Gaussian random matrix. One can formulate the problem as the uncons:
% \begin{equation}
%       \min\limits_{ \mathbf{x}\in \mathbb{R}^n} \ \ \left\{\lambda \frac{\|\h w \odot \h x\|_1}{ \|(\mathbf{1}-\h w)\odot\h x\|_2} + \frac{1}{2} \| A\h x- \h b\|_2^{2} \right\},
% \end{equation}
% where $\alpha$ is the regularization parameter, and the measurements $\h b$ is mixed with 
% the Gaussian noise with $\sigma = 0.1$.
Similar to the noise-free case, we first use a DCA-type scheme to split the optimization problem to be two-part:
\begin{equation}
    \begin{split}
    % &\min _{\mathbf{x} \in \mathbb{R}^{m}} F(\mathbf{x}) = G(\mathbf{x})-H(\mathbf{x})\\
    &G(\h x) = \alpha \|\h x\|_1 + \frac{1}{2}\| A\h x- \h b\|_2^{2}\\
    &H(\h x) = \alpha \|\h x\|_1-\lambda R_{\h w}(\h x),
    \end{split}
\end{equation}
where $\alpha$ is the parameter about the degree of convexity in each term.
% \begin{equation}
%     \begin{split}
%         \partial H(\h x) &= \alpha \partial \|\h x\|_1-\lambda\frac{\partial \|\h w \odot \h x\|_1}{ \|\left(\bf{1}-\h w \right)\odot\h x\|_2}+\lambda
%         \frac{ \|\h w \odot \h x\|_1 (1-\h w)^2\odot \h x}{ \|\left(\bf{1}-\h w \right)\odot\h x\|_2^3},
%     \end{split}
% \end{equation}
Here $\partial H(\h x)$ is exactly the same as \eqref{eq:partialH_noisefree} in  the noise-free case except we have a regularization parameter for the sorted penalty.  Note that through the DCA-type splitting, we only need to solve the unconstrained quadratic programming problem,    
\begin{equation}\label{eqdca21}
   \h x^{k+1} =  \argmin _ {\h x\in \mathbb{R}^{n}} \ \alpha \|\h x\|_1 + \frac{1}{2} \| A\h x-  \h b\|_2^{2}-\left < \h x, \h y^k \right>,
\end{equation}
where $\h y^k = \partial H(\h x^k)$.
% \begin{equation}
%     \partial H(\h x) = \alpha\operatorname{sign}(\h x^k)-\frac{\lambda\operatorname{sign}(\h w \odot \h x^k)}{ \|\h x^k-\h w\odot\h x^k\|_2}+\frac{\lambda \|\h w \odot \h x^k\|_1 (1-\h w)^2\odot \h x^k}{ \|\h x^k-\h w\odot\h x^k\|_2^3}
% \end{equation}
Here we utilize the alternating direction method of multipliers (ADMM) to solve the subproblem.
% and update each variable to get the optimal solution, where the inner-product term is a hyperplane in $\Rbb^n$ space thus \eqref{eqdca21} is convex so that ADMM can guarantee the convergence result. 
% Note that although $\h y$ will change with iterations, it does not belong to the ADMM procession. Thus in every iteration of ADMM, we will treat the $\h y$ as a fixed vector. $\h y^k$ iterations only occur in the DCA algorithm circulation. 
By introducing $\mathbf{z}\in \Rbb^{n}$ as the auxiliary variable, the optimization problem can be rewritten as:
\begin{equation}
      \argmin_{\h x, \h z\in \mathbb{R}^{n}}  \alpha \|\h z\|_1 + \frac{1}{2}\| A\h x-  \h b\|_2^{2}-\left< \h z,\h y^k\right> \quad \mathrm{s.t.} \ \  \h x = \h z.
\end{equation}\label{eqsplit}

The augmented Lagrangian can be illustrated as:
\begin{equation}\label{gulagrangian}
    L_{\delta}^{k}(\h x,\h z;\h v)= \alpha \|\h z\|_1+ \frac{1}{2}\| A\h x-  \h b\|_2^{2} -\left< \h z,\h y^k\right> + \frac{\delta}{2}\left\|\h x-\h z + \h v\right\|_2^2. 
\end{equation}
Here the problem just generates two subproblems with a scaled dual variable updating the residual at the next iteration:
\begin{equation}\label{sorted_ADMM_contrainted}
\left\{\begin{array}{l}
 \h x_{j+1} := \argmin\limits_{\h x}L_{\delta}^{k}(\h x,\h z_{j};\h v_{j}),\\
\h z_{j+1} := \argmin\limits_{\h z}L_{\delta}^{k}(\h x_{j+1},\h z;\h v_{j}),\\
\h v_{j+1} := \h v_{j}+\h z_{j+1}-\h x_{j+1},\\
\end{array}\right.
\end{equation}
where the subscript $j$ represents the inner loop index, as opposed to the superscript $k$ for outer
iterations. 
For the subproblem of $\h x$, since \eqref{gulagrangian} is a quadratic function, the minimization of subproblem $\h x$ can be reformulated
as the proximal, then the closed-form solution will be
\begin{equation}
    \h x_{j+1} = (A^{\top}A + \delta I)^{-1}(A^{\top} \h b+\delta (\h z_j+\h v_j) ).
\end{equation}
Then the subproblem of solving $\h z$ is the solution of the soft-threshold operator,
 \begin{equation}
     \h z_{j+1} =\operatorname{shrink}\left(\h x_{j+1}- \h v_{j}+ \h y^k, \frac{\alpha}{\delta} \right),
 \end{equation}
 where $$
    	\operatorname{shrink}(\h x,a)_i= 
	\begin{cases}
	x_i+a	&	 x_i<-a, 
	\\
	0	&	 -a \leq x_i \leq a, 
	\\
	x_i-a	&	x_i>a.
	\end{cases}
$$
 The last multiplier can be updated with the residual of $\h z$ and $\h x$ at iteration $j+1$. 
 \begin{equation}
     \h v_{j+1} = \h v_j+ \h z_{j+1} - \h x_{j+1}.
 \end{equation}
The algorithm is summarized as Algorithm~\ref{algorithm: noisy ADMM}. 

\begin{algorithm}
     \caption{The sorted  $L_1/L_2$ minimization via DCA-type scheme  in the case of the noisy case. }
     \begin{algorithmic}[0]\label{algorithm: noisy ADMM}
         \STATE{{\bf Input:} $A\in \mathbb{R}^{m\times n}, \h b\in \mathbb{R}^{m\times 1}$, kMax, jMax, $ \epsilon \in \mathbb{R}$, and $\rho\in R$
         % Index $i$ denotes the sorted index of $\h x_{L1}$: setting index of the first $K$-th largest magnitude of elements of $\h x_{L_1}$.
         }  
         % \STATE{Compute the $L_1$ model and obtain the restored solution $\h x_{L_1}$. }
         \STATE{{\bf Initialization: } $k = 1$, and solve for the $L_1$ minimization to get $\h x^0$ }
         % Every model started with $L_1$ solution.  
          % $ W$ started with $ I$.}
         \WHILE{$k < $ kMax or $\|\h x^{k}-\h x^{k-1}\|_2/\|\h x^{k}\| > \epsilon$}
%          \STATE{$ w_{i}=\left\{\begin{matrix}
%  1& i \notin \Gamma_{x,t} \\
%  e^{\frac{-r(t-i)}{t}}& otherwise \\
% \end{matrix}\right. $}
\STATE{$\h y^k =\alpha\operatorname{sign}(\h x^k)-\frac{\lambda\operatorname{sign}(\h w \odot \h x^k)}{ \|\h x^k-\h w\odot\h x^k\|_2}+\frac{\lambda \|\h w \odot \h x^k\|_1 (1-\h w)^2\odot \h x^k}{ \|\h x^k-\h w\odot\h x^k\|_2^3}$}
\WHILE{$j< $jMax or  $\|\h x_{j}-\h x_{j-1}\|_2/\|\h x_{j}\| > \epsilon$}
        % \STATE{Update $ W = \mathrm{diag}(\h w)$ according to $\h x^k$}
         % \STATE{$\h y^{(k+1)} = \alpha \operatorname{sign} (\h x) - \lambda\frac{ W \operatorname{sign}(\h x) }{\left\|( I- W)\h x \right\|_2}+\lambda\frac{\left\| W \h x \right\|_1(\h I- W)^2\h x}{\left\| ( I- W)\h x\right\|_2^3}$}
              % \WHILE{$j =1:2$ or $\|\h x^{k}-\h x^{(k-1)}\|_2/\|\h x^{k}\| > \epsilon$  }
         \STATE{$\h x_{j+1} = (A^{\top}A + \delta I)^{-1}(A^{\top} \h b+\delta (\h z_j+\h v_j )$}
         \STATE{$\h z_{j+1} =\operatorname{shrink}\left(\h x_{j+1}- \h v_{j}+ \h y^k, \frac{\alpha}{\delta} \right)$}
         \STATE{$\h v_{j+1} = \h v_j+ \h z_{j+1} - \h x_{j+1}$}
         \STATE{$j = j+1$}
         % \STATE{$\h z^{(k+1)} = \operatorname{shrink}\left(\h x_{U}^{(k+1)} + \frac{\h w^{k}}{\rho}, \frac{1}{\rho \|\h y^{(k+1)} \|_2}\right) $}
               % \ENDWHILE
       %  \STATE{$\h w^{(k+1)} = \h w^{k}+\rho(\h x_{U}^{(k+1)}-\h z^{(k+1)})$}
           \ENDWHILE
           \STATE{$\h x^{k+1} = \h x_j$}
            \STATE{Update the weight $\h w$ based on $\h x^{k+1}$}
         \STATE{$k = k+1$}
         \ENDWHILE
         \RETURN $\h x^{k}$ \end{algorithmic} 
\end{algorithm}

\section{Numerical results} \label{section:numerical experiment}
In this section, we will demonstrate the performance of the proposed methods via a series of numerical experiments.
All the numerical experiments are conducted on a standard laptop with CPU(AMD Ryzen 5 4600U at 2.10GHz) and MATLAB (R2021b).

% As dicussed in \Cref{toy example}, we use \eqref{eq:weight_vector} as our weight vector and update it according to $\h x^k$ in each iteration. We tune the parameters in a trial-and-error way to determine the optimal parameter set that yielded best performance. 

% \subsection{Sparse recovery} \label{sparse recovery}
In the noise-free case, we test two types of special sensing matrices, oversampled discrete cosine transform (DCT), and Gaussian random  matrix. For over-sample DCT experiments, we follow the works \cite{DCT2012coherence,louYHX14,yinLHX14} to define  $A= [\h a_1, \h a_2, \dots, \h a_n]\in \mathbb{R}^{m\times n}$ with 
\begin{equation}\label{eq:oversampledDCT}
\h a_j := \frac{1}{\sqrt{m}}\cos \left(\frac{2\pi \h h j}{F} \right), \quad j = 1, \dots, n,
\end{equation}
where $\h h$ is a random vector independently sampled and uniformly distributed in $[0, 1]^m$. Here $F\in Z^+$ controls the coherence, i.e., with $F$ increasing, the column of the sensing matrix becomes more coherent. Matrix with high coherence results in an ill-posedness, in which one can see that the coherence is generally describing the multicollinearity of matrix $A$. 
On the other hand,
% the generation and test are followed the \cite{rahimi2019scale}. 
we
generate a Gaussian random matrix by using $\mathcal N(\h 0,\Sigma)$, where $\Sigma = \{(1-R)*I(i=j)+R\}_{i,j}$ with a positive parameter $R$. Note that a larger value of $R$ will lead to greater difficulty in recovering the sparsest solution \cite{zhangX18}. Besides, we generate the support index of the ground truth $\h {\Bar{x}} \in \Rbb^n$ by the unified distribution with entries following the standard normal distribution. Then we normalize the whole ground truth to enforce its maximum to be 1. The measurements $\h b$ can be generated naturally by the product of sensing matrix $A$ and the ground truth $\h x$. All tests on sparse recovery obey the same environment setting. We set the size of sensing matrices $64\times 1024$ to test the performance with a severely ill-posed problem. In addition, in the previous work \cite{doi:10.1137/110838509}, they have shown there is a minimum separation between these adjacent support index of $\bar{\h x}$, recovering the sparsest solution could be possible. Here throughout the noise-free case, we impose all the models satisfying that:  $\min _{i, i^{\prime} \in \operatorname{supp}(\overline{\mathbf{x}})}\left|i-i^{\prime}\right| \geq 1$ for both the oversampled-DCT and Gaussian sensing matrices.
For the measure of performance, we take the relative error (ReErr) between the reconstruction solution $\h x$ and the ground truth $\bar{\h x}$, defined as $ \|\h x-\overline{\h x}\|_{2} /\|\overline{\h x }\|_{2} $. Furthermore,
 Now we define the success rate: the number of successful trials over the total number of trials. If ReErr $< 10^{-3}$, then we treat it as a successful trial to recover the ground truth.
We adopt a commercial optimization software  Gurobi \cite{optimization2014inc} to minimize the $L_1$ norm via linear programming for the sake of efficiency. We restrict the values of $\h x^+$ and $\h x^-$ within $[0, 1]$,  owing to the normalization. The stopping criterion is when the relative error of  $\h x^{k}$ to $\h x^{k-1}$ is smaller than  $10^{-8}$ or iterative number exceeds $10n$.

The experiments in the noisy case follow a setup in \cite{6205396}.
We focus on recovering a signal $\h x$ of length $n = 512$ with $s = 130$ nonzero elements from $m$ measurements, represented by $\h b$, obtained through a Gaussian random matrix $A$. The matrix $A$ has normalized columns with a zero mean and a unit Euclidean norm, and Gaussian noise with zero mean and standard deviation $\sigma = 0.1$ is also taken into account. Fewer $m$ leads to a more difficult and ill-posed problem for the reconstruction. We use the mean-square error (MSE) metric to assess the recovery performance. 
Since the sorted $L_1/L_2$ model is a nonconvex-type model, the choice of initial guess $\h x^{(0)}$ directly impacts the performance. We use the restored solution via the $L_1$ minimization as the initial guess throughout the noise-free and noisy case.

\subsection{Discussion on weight vector}
We consistently utilize the same weight vector as \eqref{eq:weight_vector} throughout all the numerical experiments. 
% \begin{equation}
%     w_i= 
% 	\begin{cases}
% 	1	&	\text{otherwise}, 
% 	\\
% 	e^{-r(t-i)/t}	&	 i\in \Gamma_{\h x,t}.
% 	\end{cases}
% \end{equation}
Here $t$ controls the cardinality of $\Gamma_{\h x,t}$, 
% (actually the cardinality is $t-1$, for simplicity we omit the difference), 
% i.e., the number of elements $ w_i \neq 1$. 
and $r$ is the ``slope rate'' to determine the shape of  $\h w$, where a large $r$ clearly leads to a large curvature. 
Due to the uncertainty of the performance of support detection by $L_1$ initial guess,  we choose $t$ around ${\|\h x_{L_1}\|_0}/{3}$ which is similar to one in \cite{huang2015nonconvex}. Here $\h x_{L_1}$ is the restored signal via $L_1$. 
Furthermore, it is noteworthy that since the weight is an exponential function, a large $r$ will cause a significant proportion of $\left\{w_i \right\}, i \in \Gamma_{\h x,t}$ values clustering at very small values.
Hence, selecting a small value for $r$ is a suitable choice at the beginning. 
% , which if we sort the $w_i$ by its increasing magnitudes it will be like 'heavy-left-tailed'. 
%  From the perspective of designing the weight, we want the largest scale elements in the front part to be as distinguishable as possible. 
% Hence, selecting a smaller value for $d$ is a suitable choice.
In addition, in Figure~\ref{fig:toy2}, one can observe that a small value of $r$ can amplify the scale of local optimums and makes them sharper, which can improve the probability of the model encountering the local optimum.
% Therefore, we choose the slope rate $r=0.8$ throughout the numerical experiments.

% For selecting the number $t$, an intuitive method is that we aim to ensure the sparsity, and due to the uncertainty of the performance of support detection by $L_1$ initial guess, $t$ should be larger for sparsity. Thus, we choose $t$ around ${\|\h x_{L_1}\|_0}/{3}$ similar with \cite{huang2015nonconvex} where $\h x_{L_1}$ is the restored signal via $L_1$. In the presence of noise, the ability of support detection of $L_1$ becomes more uncertain. Therefore, a large $t$ can lead to a more uniform dispersion of adjacent element scales.
% , which is appropriate for this situation.

Furthermore, 
% from \Cref{fig:obj value down}, one can observe that the value of $\h w$ tends to converge or fluctuate around specific values after a certain number of iterations, which indicates we need to update the values of the weight function after it has stopped updating.
% Thus a very natural thought is to adjust the values when the solution converges. W
we devise a two-stage weighting scheme. 
In the first stage, we use small values of $t$ and $d$ to encourage finding more entries in the support. After the 20 iteration step,  
 we alter the value of the weight function to accelerate the support detection and the convergence. To be precise, 
in the absence of noise, we employ  $(t, r)=(20, 0.8)$ in the first stage and $(t, r)=(27, 3)$ in the second stage, whereas, for the noisy case, we choose $(t,r) = (100, 0.8)$ and $(t,r)=(120,3)$ in the respective stages. Figure~\ref{fig:weight} shows the noise-free case, where the left sub-figure is the shape of the weight vector sorted in ascending order. 
  Note that changing the value of $t$ not only modifies values of the weight function but also implies a desire to encompass more possible support sets that might have been overlooked. Figure~\ref{fig:weight} shows that the two-stage scheme has higher success rates for sparse recovery than the one-stage, especially when the sparsity level is high.
\begin{figure}[h] 
		\begin{center}
			\begin{tabular}{cc}
			    % Relative error &   Objective value \\
				\includegraphics[width=0.45\textwidth]{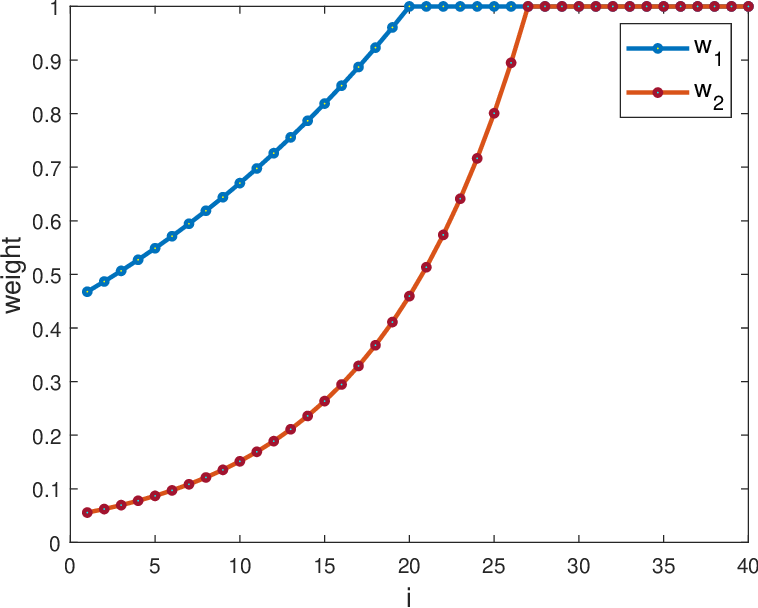}&			    \includegraphics[width=0.45\textwidth]{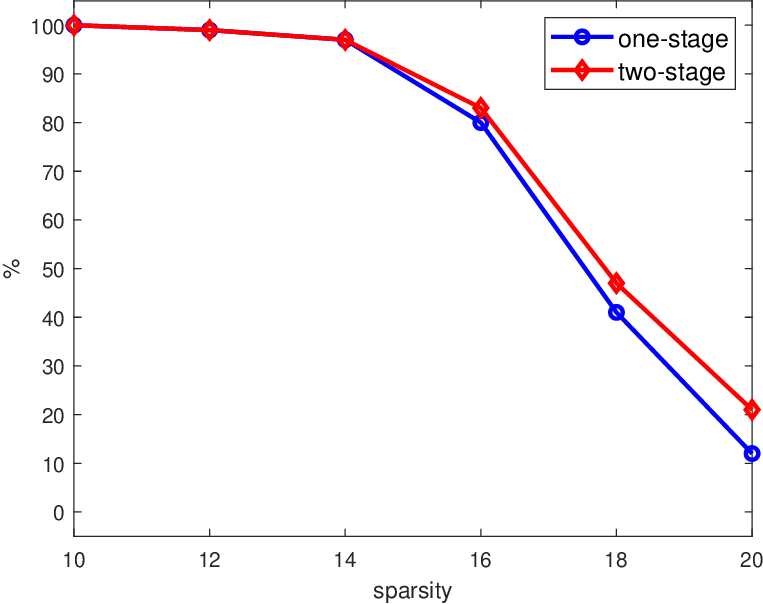}
			\end{tabular}
		\end{center}
		\caption{Left: two weight functions sorted in an ascending order, where $\h w_1$ and $\h w_2$ are with the parameters $(t,r) = (27,3)$, and $(27, 0.8)$, respectively. Right: comparison of success rate with one-stage and two-stage sorted $L_1/L_2$ in the noise-free case under the Gaussian random matrix with $R=0.1$. }
            \label{fig:weight}
	\end{figure}

% From the \Cref{}, the graph clearly shows that the use of two-stage weights can significantly reduce the value of the objective function.
% \begin{figure}
%     \centering
%     \includegraphics{}
%     \caption{Caption}
%     \label{fig:my_label}
% \end{figure}

\subsection{Algorithm behaviors}

We conduct a numerical experiment to empirically demonstrate the convergence of the proposed algorithms. In the noise-free case, 
the relative error and objective value decrease steeply in the first 10 iterations, as shown Figure~\ref{fig:obj value down}.  The relative error reaches to the computation accuracy $10^{-15}$ in the 9-th iteration.  
% This implies the DCA and Gurobi could efficiently solve the noise-free problem in just a few iterations.

\begin{figure}[h] 
		\begin{center}
			\begin{tabular}{cc}
			    % Relative error &   Objective value \\
				\includegraphics[width=0.45\textwidth] {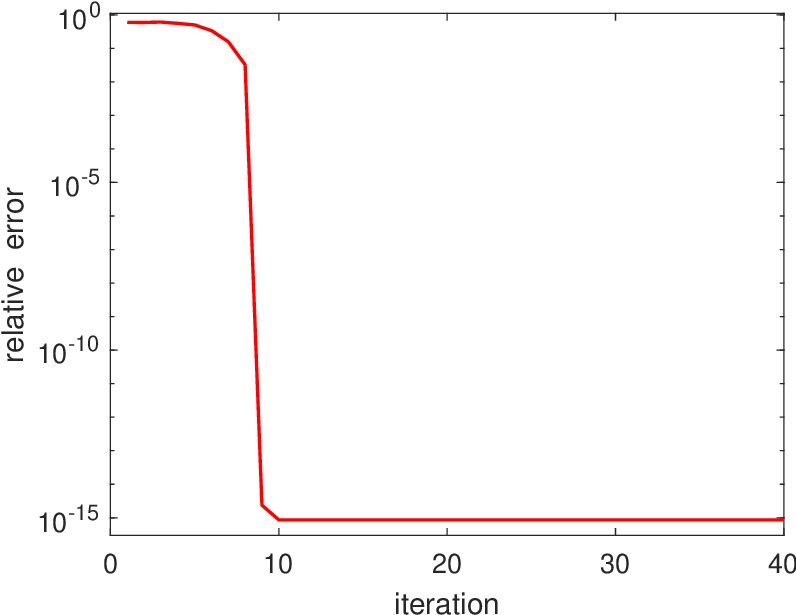}&
			    \includegraphics[width=0.45\textwidth]{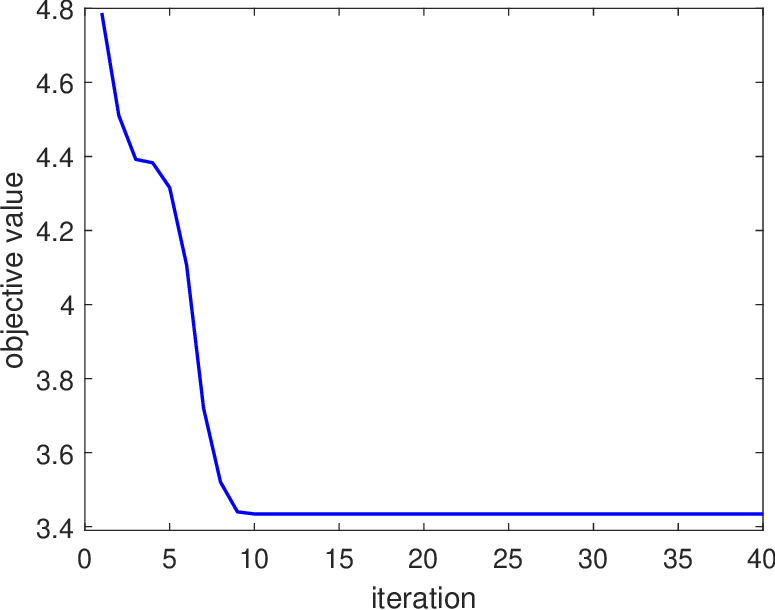}
			\end{tabular}
		\end{center}
		\caption{The relative error (left) and objective value (right) for empirically demonstrating the fast convergence of the proposed algorithm for the noise-free model.
		}\label{fig:obj value down}
	\end{figure}

Regarding the noisy case, Figure~\ref{fig:noisy_obj_re} shows the change of the objective value and relative error with different maximum iterations in the inner loop. Here maximum iteration for the outer one is set as 100. Note that with the inner iteration increasing, curves of objective value and relative error both have larger curvatures, which empirically show a more rapid convergence rate. Meanwhile, only one inner iteration will cause the objective value to rise steeply and then decrease.
% , which shows the lack of enough inner iterations.
Besides, although a large number of inner iterations will result in a more rapid convergence, it takes much more computational time.
% , which leads to a very similar result for the noisy signal's sparse recovery. 
We observe that 
20 iterations in the inner loop are sufficient to yield good results. Hence, in the following experiments, we fix the maximum number of the inner iteration as 20. 

\begin{figure}[h]
		\begin{center}
			\begin{tabular}{cc}
                     % $inner\quad iteration=1$  & $inner \quad iteration=2$ \\
                     % Relative error &  Objective value \\
			    \includegraphics[width=0.45\textwidth]{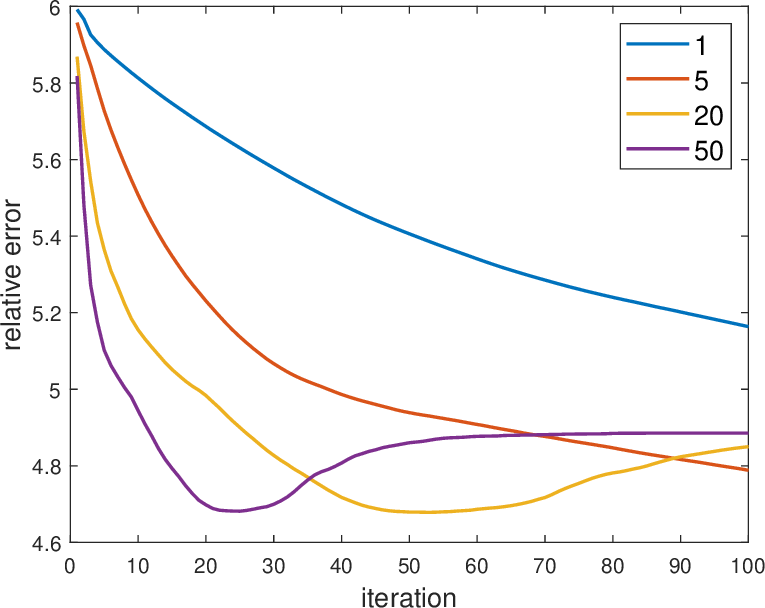}&
				\includegraphics[width=0.45\textwidth]{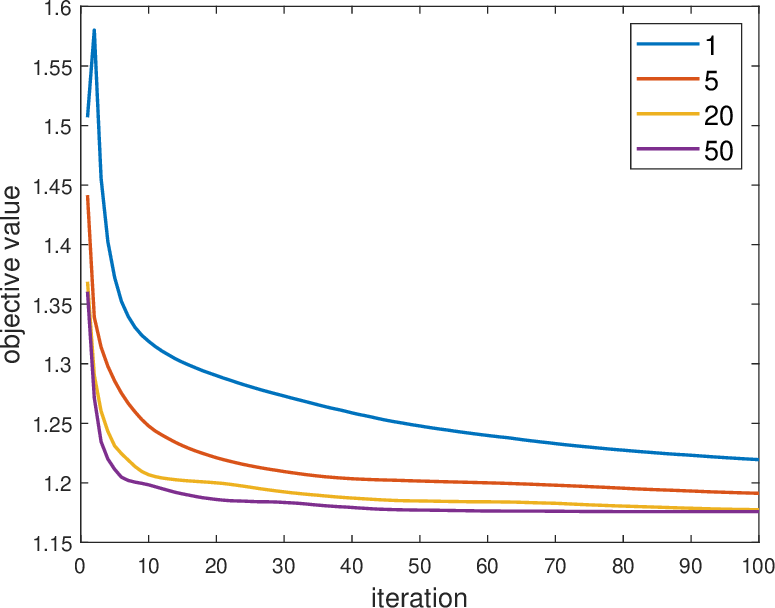}
			\end{tabular}
		\end{center}
		\caption{The relative error and objective value with different inner iterations $1, 5, 20, 100$ for empirically demonstrating the convergence with the noisy model.}\label{fig:noisy_obj_re}
            \end{figure}

% Here $L_1$ and $L_1/L_2$ are solved by ADMM. Sorted $L_1/L_2$ solved by DCA and ADMM with regularization parameters $\lambda = 0.05$, $\alpha=0.08$. The ADMM inner iteration is 30 and outer iteration is 60.

\subsection{Comparison on various models} \label{subsection:compare with other models}
First, we consider the noise-free case and compare the proposed sorted $L_1/L_2$ minimization with four models for sparse recovery: the $L_1$ minimization, the $L_p$ model\cite{chartrand07}, the $L_{1}$-$L_{2}$ model \cite{zibulevsky2010l1,lou2018fast} and the $L_1/L_2$ minimization \cite{rahimi2019scale}. 
% Here $L_1$ is solved by DCA with Gurobi; $L_p$ by IRLS algorithm with $p=0.5$; $L_1/L_2$ is solved by ADMM. 
% The parameter settings are followed by their works. 
% We simulate the sparse recovery with the oversampled-DCT matrix with $F=5$, $F=10$ and Gaussian matrix with $r=0.1$, $r=0.8$. \jj{Here for simplicity, we fix the regularization parameter in DCA to be $1$.}
Figure~\ref{fig:osdct}  reveals that the sorted $L_1/L_2$ model makes the state-of-art performance with the oversampled-DCT matrix, especially when $F=5$ with relatively low coherence. For a relatively low coherence, the $L_{1}$-$L_{2}$ minimization makes the third best, but it makes the best when $F=10$, while the sorted $L_1/L_2$ makes the second best. The $L_1/L_2$ model makes the second best and achieves a similar performance compared to the sorted $L_1/L_2$ model when $F=10$. The $L_p$ method does not perform well in the high coherence $F=10$, which is even worse than $L_1$. 
Regarding the Gaussian sensing matrix case,  Figure~\ref{fig:guassian} shows that $L_p$ is very excellent. We can clearly see the sorted $L_1/L_2$ model gets comparable results compared to the $L_p$ model both for $R=0.1$ and $R=0.8$. While $L_{1}$-$L_{2}$ performs well using oversampled-DCT in high-coherence situations, its performance worsens in the Gaussian matrix case. It is noteworthy that 
we use the same settings and parameters across different matrices or coherence levels and the robustness of the sorted $L_1/L_2$ model. 
\begin{figure}[h]
		\begin{center}
			\begin{tabular}{cc}
			   % (a) Toy example 1 $(a = -3.5) & Toy example 1 $(a = -3.5) \\
			   $F=5$  & $F=10$ \\
		       \includegraphics[width=0.45\textwidth]{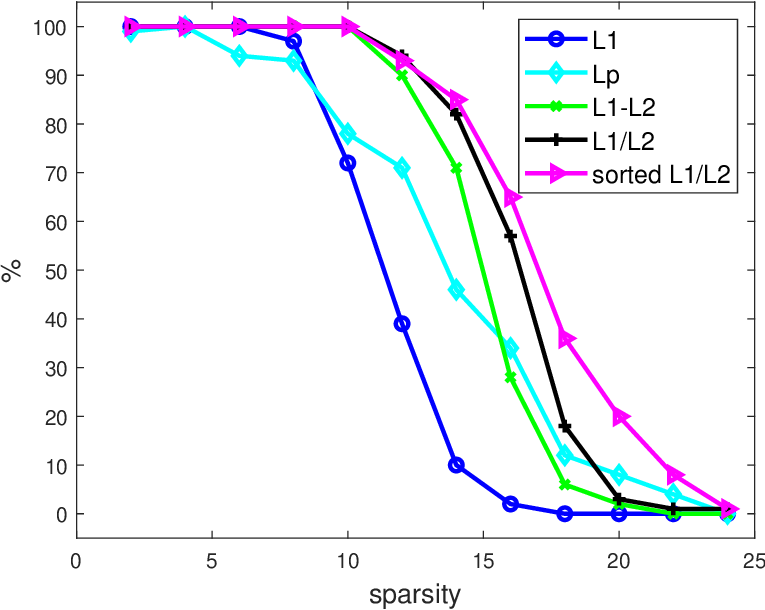} &
		       \includegraphics[width=0.45\textwidth]{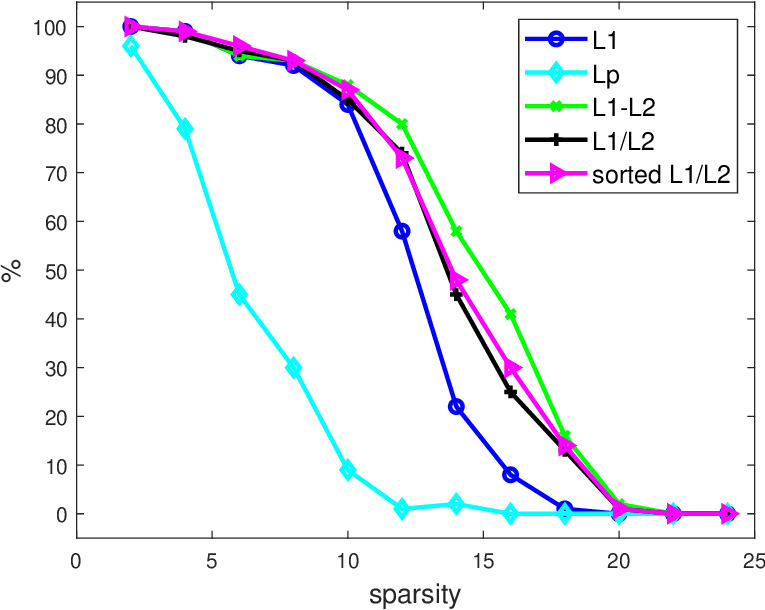}  
			\end{tabular}
		\end{center}
		\caption{Success rates for different models under the oversampled-DCT matrix.}
		\label{fig:osdct}
	\end{figure}
 \begin{figure}[h]
		\begin{center}
			\begin{tabular}{cc}
			   % (a) Toy example 1 $(a = -3.5) & Toy example 1 $(a = -3.5) \\
			   $R=0.1$  & $R=0.8$ \\
				\includegraphics[width=0.45\textwidth]{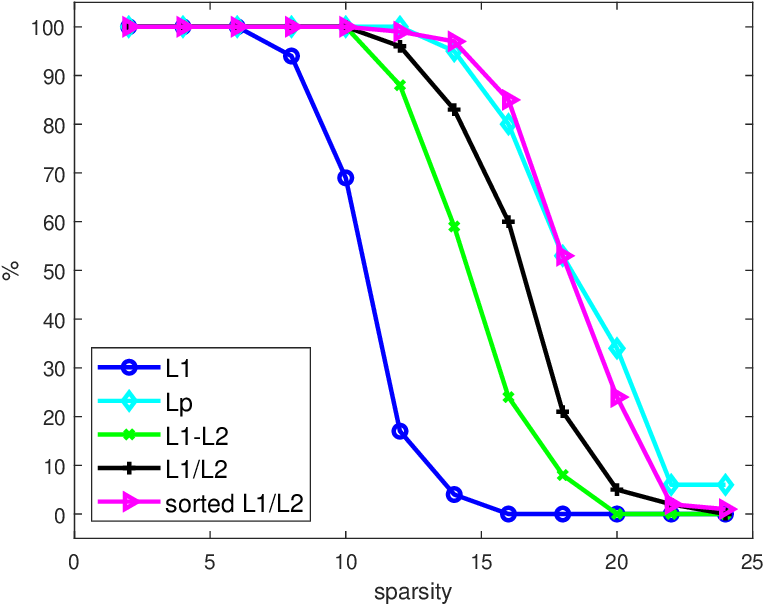} &
			    \includegraphics[width=0.45\textwidth]{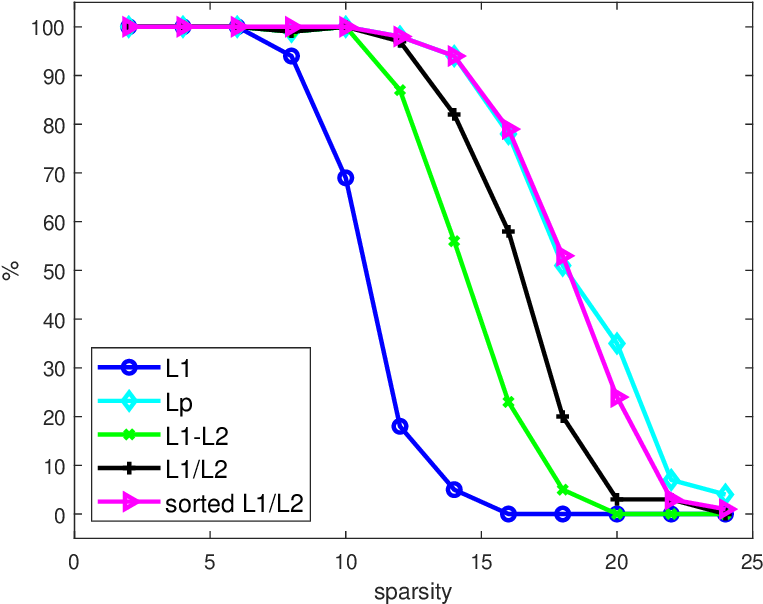}  
			\end{tabular}
		\end{center}
		\caption{Success rates for different models under the Gaussian matrix.}
		\label{fig:guassian}
	\end{figure} 
 
% \begin{figure}[h]
% 		\begin{center}
% 			\begin{tabular}{cc}
% 			   % (a) Toy example 1 $(a = -3.5) & Toy example 1 $(a = -3.5) \\
% 			   $F=5$  & $F=10$ \\
% 				\includegraphics[width=0.42\textwidth]{fig/ostF=5.eps} &
% 			    \includegraphics[width=0.42\textwidth]{fig/osdctF=10.eps}
% 			\end{tabular}
% 		\end{center}
% 		\caption{Success rates for L1/L2 and its variants models with DCA in the oversampled DCT case
% 		}
% 	\end{figure}	
% 	%% 插入的文件名不能有空格
 
% \begin{figure}[h]
% 		\begin{center}
% 			\begin{tabular}{cc}
% 			   % (a) Toy example 1 $(a = -3.5) & Toy example 1 $(a = -3.5) \\
% 			   $r=0.1$  & $r=0.8$ \\
% 				\includegraphics[width=0.42\textwidth]{fig/Task2_Guassian_r1.eps} &
% 			    \includegraphics[width=0.42\textwidth]{fig/Task2_Guassian_r2.eps}
% 			\end{tabular}
% 		\end{center}
% 		\caption{Success rates for L1/L2 and its variants models with DCA in the Gaussian case
% 		}
% 	\end{figure}

Now we consider the noisy case and compare the proposed sorted $L_1/L_2$ model with other models in the noisy case:  $L_1$, $L_{1}$-$L_{2}$ \cite{zibulevsky2010l1}, $L_p$ via the half-thresholding method \cite{6205396}, error function for sparse recovery (ERF) \cite{guo2021novel} and the oracle performance in recovering a noisy signal. In addition, we implement and compare the $L_1/L_2$ minimization in the unconstrained formulation via a DCA-type scheme. 
% As talked before, we use mean-square error between the ground truth $\Bar{\h x}$ and recovered signal $\h x$, $\|\h x-\bar{\h x}\|_2$ to measure the performance. 
If the ground truth of support set of  $\h s = \operatorname{supp}(\h x)$ is known, which is the index set of nonzero entries in $\Bar{\h x}$, then we can give the ordinary least square (OLS) solution. Thus we take the MSE of OLS as an oracle performance with $ \sigma \mathrm{tr}(A_{s}^{\top}A_{s})^{-1}$, as the benchmark. 
% And we set the $L_1$ solution as the initial guess for $L_1/L_2$ and the sorted $L_1/L_2$ minimization.

% \jj{parameters setting}
Regarding parameter settings for the noisy sparse recovery, 
% our model has a regularization parameter of $\lambda$ for $L_1$ term due to the dc decomposition, a regularization parameter of $\alpha$ for the sorted penalty term, and a Lagrange multiplier $\delta$ in the ADMM subproblem solved during the inner iteration. To ensure convergence during the inner loop, 
we set $\delta=0.8$ and $\alpha=0.1$ in our model. As for $\lambda$, we observe that a smaller value leads to better performance for smaller $m$ (which corresponds to more challenging problems), while a larger value of $\lambda$ is required for larger $m$ (which corresponds to easier problems) to increase the constraint on the penalty term and prevent over-fitting.
Since earlier research has indicated that selecting a fixed penalty parameter that is either too small or too large may lead to a substantial increase in computational expenses. 
Here we implement the regularization parameter $\lambda$ for an adaptive updating strategy resembling the related topic \cite{hu2012fast}.
Note that a similar parameter setting ($\lambda = 0.1 \times m/270 $) is used for $L_1$, while we empirically use $m^2/n^2-0.1$ for our model as the number of rows $m$ and columns $n$ is given. 
It should be noted that this parameter setting is empirical, and the performance is not sensitive to small oscillations with $\lambda$. 
% Due to computational efficiency, we did not perform cross-validation to adjust the optimal parameters for each value of $m$. 
For other models, we follow the setting from the previous work except for $L_1/L_2$, in which we adopt the DCA with parameters $\lambda = 0.06, \alpha = 0.05$, and $\delta =0.9$. The initial guess for all models is the solution obtained from $L_1$.

In Figure~\ref{fig:denoising} and Table~\ref{Tab:table:noisy signal recovery}, the measurements are taken at 10 intervals ranging from 250 to 360. The oracle performance is based on the OLS given the support of the ground truth, which is extremely difficult to achieve.
However, our proposed model closely approximates the oracle performance, especially with a large number of rows $m$, and achieves state-of-the-art performance with any $m$. 
As $m$ decreases, the problem becomes more ill-posed, and the performance of all models starts to converge the same.
% \begin{figure}[h]
% 		\begin{center}
% 			\begin{tabular}{cc}
% 			   % (a) Toy example 1 $(a = -3.5) & Toy example 1 $(a = -3.5) \\
% 			                                              \\
% 				\includegraphics[width=0.55\textwidth]{fig/compare.eps} 		    
% 			\end{tabular}
% 		\end{center}
% 		% \caption{MSE of sparse recovery in the noisy case by different models.}
% 		\label{fig:denoising}
% 	\end{figure}
\begin{figure}[h]
		\begin{center}
			\begin{tabular}{cc}
			   % (a) Toy example 1 $(a = -3.5) & Toy example 1 $(a = -3.5) \\
			                                              \\
				\includegraphics[width=0.70\textwidth]{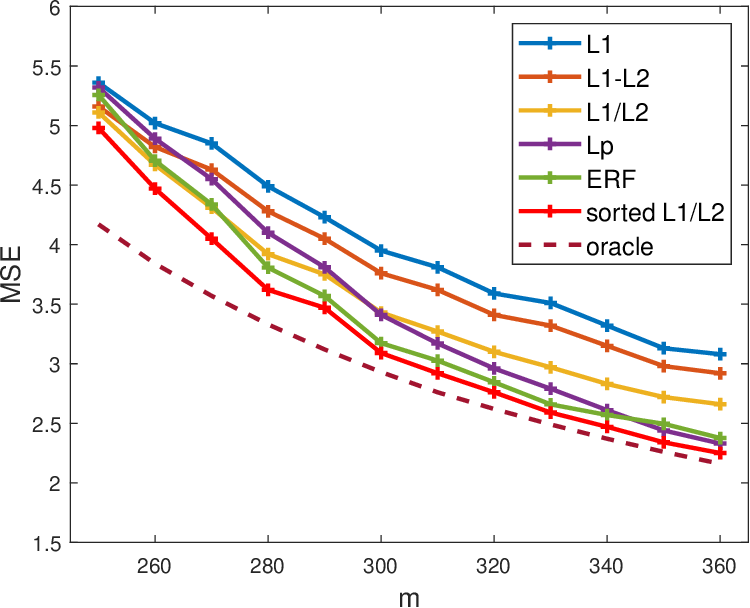} 		    
			\end{tabular}
		\end{center}
		\caption{MSE of sparse recovery in the noisy case by different models.
		}\label{fig:denoising}
	\end{figure}
\begin{table}[h]
		\begin{center}
			\scriptsize
			\caption{MSE of sparse recovery in the noisy case. } 
		\begin{tabular}{l|cccccccccccc} 
				\hline 
				% \multirow{2}{*}{Image} & \multirow{2}{*}{Line} & \multicolumn{2}{c|}{ZF}& \multicolumn{2}{c|}{$L_1$ } & \multicolumn{2}{c|}{$L_p$ } & \multicolumn{2}{c|}{$L_1$-$\alpha L_2$ } & \multicolumn{2}{c}{$L_1/L_2$ }  \\ \cline{3-12} 
				% &  & PSNR &   RE & PSNR &   RE & PSNR & RE & PSNR &  RE & PSNR &   RE  \\ \hline
			$m$  & 250 & 260 & 270 &280 &290&  300 \\ \hline
		    $L_1$ &  5.36 & 5.02 & 4.85 & 4.49 & 4.23 & 3.95  \\
                $L_{1}$-$L_{2}$ & 5.16   & 4.82   & 4.63   & 4.28   & 4.05   & 3.76    \\
                $L_1/L_2$  & 5.11 & 4.67 & 4.31 & 3.92 & 3.75 & 3.43 \\
                $L_p$ & 5.32   & 4.89   & 4.55   & 4.10   & 3.81   & 3.41   \\
                ERF & 5.26   & 4.70   & 4.34   & 3.81   & 3.57   & 3.17  \\
		      sorted $L_1/L_2$ & \bf4.98   & \bf4.50   & \bf4.06   & \bf3.63   & \bf3.47   & \bf3.10       \\ \hline
	       	oracle  & 4.17 & 3.84 & 3.57 & 3.33 & 3.12 & 2.93     
				 \\ \hline 	\hline
                 $m$  &310 &320 & 330 &340 &350&   360 \\ \hline
		    $L_1$  & 3.81 & 3.59 & 3.51&  3.32 & 3.13 & 3.08 \\ 
                $L_{1}$-$L_{2}$ & 3.62   & 3.41   & 3.32   & 3.15   & 2.98   & 2.92 \\
                $L_1/L_2$  & 3.27 & 3.10 & 2.97 & 2.83 & 2.72 & 2.66 \\
                $L_p$ & 3.17   & 2.96   & 2.79   & 2.61   & 2.44   & 2.33 \\
                ERF  & 3.03   & 2.84   & 2.66   & 2.57   & 2.50   & 2.38  \\
		      sorted $L_1/L_2$  & \bf2.90   & \bf2.73   & \bf2.58   & \bf2.46   & \bf2.32   & \bf2.24    \\ \hline
	       	oracle  & 2.76 & 2.62 & 2.49 & 2.37 & 2.26 & 2.16 \\
                \hline
			\end{tabular}\label{Tab:table:noisy signal recovery}
			\medskip
		\end{center}
	\end{table}

\subsection{Support detection} \label{section: support detection}
In this section, we will conduct a series of numerical experiments to state the ability to detect support in different models. 
% Besides, we design a novel trial to detect the true support by only considering limited indexes with large absolute values with a high probability.
% For many situations, sparse recovery is an exceedingly challenging task for every regularization model. 
% It requires not only finding the support index of the ground truth $\bar{\h x}$, but also demands recovering the value of the entry, which is more difficult as the sparsity increases.
% Even though the numerical experiment of sparse recovery in \Cref{fig:osdct,fig:guassian}  shows the proposed sorted $L_1/L_2$ minimization gives a rather high probability to achieve the restoration, one can not truly know whether the measurements can be exactly recovered without ground truth in such high sparsity case. 
% Thus when we can not faithfully guarantee the precision of recovering the exact sparsest solution, then detecting the support or called variable selection in statistics, becomes the priority task.
The performance of detecting support is assessed in terms of the {\bf recall rate}, defined as the ratio of the number of identified true nonzero entries over the total number of true nonzero entries, and the {\bf precision rate}, defined as the ratio of the number of identified true nonzero entries over the number of all the nonzero entries obtained by the algorithm. 
% \begin{equation}
% \label{eq:rate of detected support}
%     \textbf{recall} = \frac{\textbf{supp}(\h x)\cap \textbf{supp}(\bar{\h x})}{\|\bar{\h x}\|_0}
% \end{equation}
% where $\| \bar{\h x}\|_0$ indicates the sparsity of the ground truth.
We compare the proposed model with the $L_1$, $L_{1}$-$L_{2}$, and $L_1/L_1$ minimization for sparsity 10, 12, 14, 16, 18, 20. 
% when detecting support in sparsity 1-9 is trivial for current regularization models.
Note that we consider the Gaussian sensing matrix $A$ with the same parameter setting as in Session~\ref{subsection:compare with other models}. Each sparsity corresponds to 100 independent repeated trials, and then we take their average to calculate the recall and precision rates.

Figure~\ref{fig support detection} exhibits the proposed sorted $L_1/L_2$ minimization achieves the state-of-art recall and precision rate compared with other models. The recall rate of the sorted $L_1/L_2$ model is higher than $L_1$ and $L_1/L_2$ minimization, which indicates its good performance in detecting the true support.  
The $L_1$  model achieves almost 50\% support index in the sparsity 20. Such performance is significantly higher than its result in sparse recovery, which implies $L_1$ can indeed detect the support but fail to recover the exact solution.  $L_1/L_2$ performs the second best.  
% The sorted $L_1/L_2$, $L_1/L_2$, and $L_1$ models generate a large number of nonzero entries when they fail to conduct the precise sparse recovery. To be more precise, the sorted $L_1/L_2$ generates the number of $m$ nonzero entries as the same as the $L_1$ minimization. For the $L_1/L_2$ model, it often generates more than $m$. 
% Thus, the rate of wrong-detecting becomes
% a crucial problem since it is typically a rare ratio of $\textbf{supp}(\h x)\cap \textbf{supp}(\bar{\h x})/{\|\h x\|_0} $, i.e. finding the true support detected by the model from the large number $\|\h x\|_0$.
The results in Figure~\ref{fig support detection} (right)  demonstrate that the sorted $L_1/L_2$ model achieves an exceptional level of precision in various sparsity. At sparsity levels ranging from 10 to 14, our model exhibited a perfect precision of nearly average 100\%, which indicates that all the non-zero elements are detected by the model. In comparison, $L_1$ achieves a precision of less than 30\% at a sparsity level of 12 and remains steady at around 15\%. The performance of $L_1/L_2$ is inferior to that of our model, as it maintained good precision at low sparsity levels. However, as the sparsity level increased, such as at the level of 18 and 20, our model demonstrates a superiority of 28\% and 17\% over $L_1/L_2$, respectively.

\begin{figure}[h]
		\begin{center}
			\begin{tabular}{cc}
			   % (a) Toy example 1 $(a = -3.5) & Toy example 1 $(a = -3.5) \\
			    % \\ Recall  & Precision \\
				\includegraphics[width=0.45\textwidth]{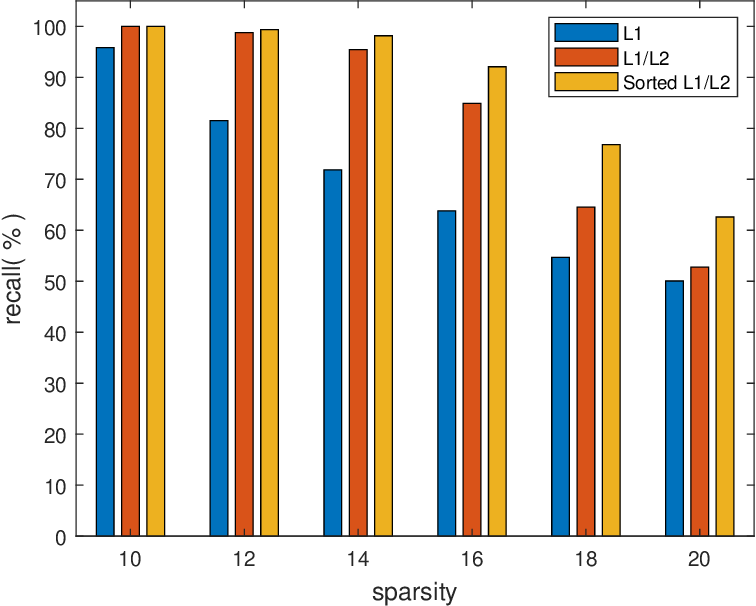} &
			    \includegraphics[width=0.45\textwidth]{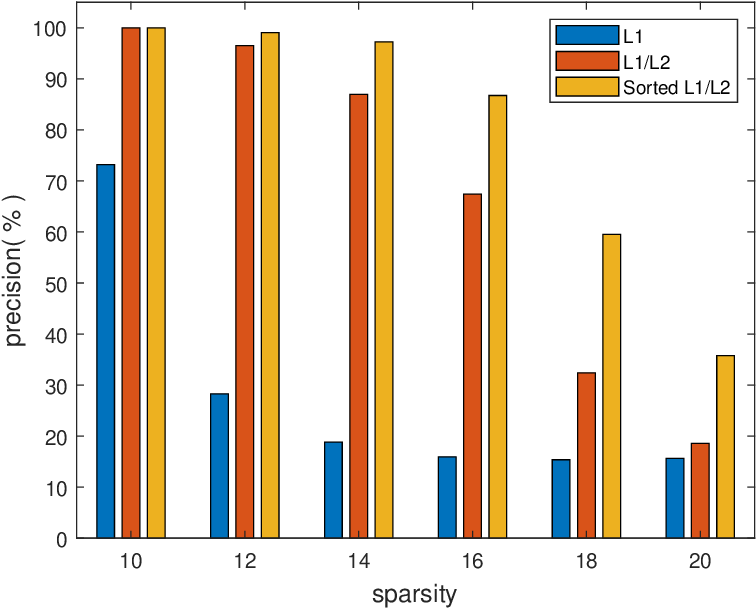}
			\end{tabular}
		\end{center}
		\caption{The recall  (left) and precision rates (right) for the $L_1$, $L_1/L_2$, and sorted $L_1/L_2$ models with sparsity from 10 to 20. 
		}\label{fig support detection}
	\end{figure}

\begin{table}[h] 
		\begin{center}
			\scriptsize
			\caption{The rate of finding support focused on the indices with the $\| \bar{\h x}\|_0$-largest magnitudes by different models with 100 independent trials. All models achieved the best when sparsity 1-9 except for $L_1$ thus we omit in the table.} \label{tab: support detection}
 \begin{tabular}{l|cccccc} 
				\hline 
				% \multirow{2}{*}{Image} & \multirow{2}{*}{Line} & \multicolumn{2}{c|}{ZF}& \multicolumn{2}{c|}{$L_1$ } & \multicolumn{2}{c|}{$L_p$ } & \multicolumn{2}{c|}{$L_1$-$\alpha L_2$ } & \multicolumn{2}{c}{$L_1/L_2$ }  \\ \cline{3-12} 
				% &  & PSNR &   RE & PSNR &   RE & PSNR & RE & PSNR &  RE & PSNR &   RE  \\ \hline
			$Sparsity$ & 10 &  12 &  14  & 16 & 18 & 20 \\ \hline
            $L_1$ &  83\% &   55\%  &    35\% &   21\% &    14\% &    11\% \\
            $L_{1}$-$L_{2}$ &  \bf{100\%} &   96\%  &    81\% &   52\% &    25\% &    17\% \\
            $L_1/L_2$ &  \bf{100\%} &    \bf{100\%} &   90\% &    72\% &    41\%  &    24\%          \\
            sorted $L_1/L_2$ &  \bf{100\%} &   98\%  &    \bf{96\%} &   \bf{87\%} &    \bf{64\%} &    \bf{44\%}   \\ \hline	 
			\end{tabular}
			\medskip
		\end{center}
	\end{table}

At last, we design an intriguing but straightforward experiment of variable selection. 
Figure~\ref{fig support detection} shows that there are a lot of nonzero entries selected by $L_1, L_1/L_2$, or the sorted $L_1/L_2$ model. Thus we can select from the indices with large magnitudes to guarantee a high probability of detecting some true support. In light of this statement,
% Define :
% \begin{equation}
%      \frac{\textbf{supp} \left(\Gamma_{\h x,\|\h x\|_0}\right)\cap \textbf{supp}\left(\Gamma_{\bar{\h x},\|\bar{\h x}\|_0}\right)}{\|\bar{\h x}\|_0}
% \end{equation}
% where $\Gamma_{\mathbf{x}, t} \subseteq\{1, \ldots, n\}$ with cardinality $t$ is a set contain the indices of the entries of
% $\h x$ with the $t$ largest magnitudes. 
 we only consider the entries with $\|\bar{\h x} \|_0$-largest magnitudes in the restored signal $\h x$. To be more precise, for example, given sparsity 20, we only focus on the 20 largest magnitudes indexes, and then we compute the rate of finding the true support index.
% Apparently $P(M_{\Bar{\h x}}) = 1$ for any sparsity, since the number of sparsity maximum absolute indexes is exactly equal to the nonzero entries of the ground truth.

% $P(M_{sort})$ indicates the proportion of $M$ by the sorted $L_1/L_2$ model.
Table~\ref{tab: support detection} shows our proposed sorted $L_1/L_2$ model achieved state-of-art performance compared with other models. For sparsity 18, our proposed model can find the average of 64\%, i.e. nearly 11 true support in 18 nonzero entries indices, while $L_1$, $L_{1}$-$L_{2}$, and $L_1/L_2$ only find 2, 4, and 7, respectively. As the sparsity level rises, our proposed model stays to find the average of 8 true support, which is two times that of the $L_1/L_2$ model.

 % The tabular \ref{tab: support detection} illustrates the sorted model can maintain better initial true support detected by $L_1$ model than $L_1/L_2$ model even in the high sparsity, noticing we set the $L_1$ solution as the initial guess for both of them. For sparsity 2-8, we omit the results since both conditional probabilities are 1. 

\section{Conclusion and future works} \label{senction:conclusion}
In this paper, we discuss a novel type of regularization by generalizing the $L_1/L_2$ model into a sorted $L_1/L_2$ scheme.  We provide a theoretical analysis of the existence of solutions. By employing the DCA-type scheme, we can achieve state-of-the-art performance for sparse recovery in both noise-free and noisy scenarios. The experimental results demonstrate that the sorted model has a significant advantage over the $L_1/L_2$ model and is capable of detecting the support set with high accuracy, even in high-sparsity cases. These sorted models can be easily incorporated with a box constraint if it is available, which ensures the boundedness of the solution in the subproblem. Note that it is possible that the solution of our algorithm can be unbounded if there is no box constraint. However, we empirically test that $\{\h x^k\}$ from the noise-free or noisy scenarios is always bounded and convergent for general random matrices $A$. 
Our future research will involve extending this sorted model to the matrix and tensor formulation to consider low-rankness instead of sparsity. Furthermore, we plan to explore its application to image processing and investigate sparsity on the gradient.

\begin{acknowledgements}
 C.~Wang was partially supported by the Natural Science Foundation
of China (No. 12201286), HKRGC Grant No.CityU11301120, and the Shenzhen Fundamental Research Program JCYJ20220818100602005.  M.~Yan was partially supported by the Shenzhen Science and Technology Program ZDSYS20211021111415025.
\end{acknowledgements}

\section*{Data Availability}
The MATLAB codes and datasets generated and/or analyzed during the current study  will be available after publication.

\section*{Declarations}
The authors have no relevant financial or non-financial interests to disclose.  The authors declare that they have no conflict of interest.


\bibliographystyle{spmpsci}      % mathematics and physical sciences
\bibliography{refer}

% % Non-BibTeX users please use
% \begin{thebibliography}{}
% %
% % and use \bibitem to create references. Consult the Instructions
% % for authors for reference list style.
% % Format for Journal Reference
% Author, Article title, Journal, Volume, page numbers (year)
% Author, Book title, page numbers. Publisher, place (year)
% \end{thebibliography}

\end{document}